\documentclass[11pt]{amsart}%
\usepackage{hyperref}
\usepackage{amsmath,latexsym,amssymb,amsthm,array,amsfonts}

\usepackage{mathrsfs,dsfont,bbm}

\numberwithin{equation}{section}

\theoremstyle{plain}

\newtheorem{theorem}{Theorem}[section]
\newtheorem{definition}{Definition}[section]
\newtheorem{lemma}{Lemma}[section]
\newtheorem{corollary}{Corollary}[section]
\newtheorem{proposition}{Proposition}[section]
\newtheorem{remark}{Remark}[section]

\begin{document}

\title[The quadratic covariation for a weighted-fBm]{The quadratic covariation for a weighted fractional Brownian motion${}^{*}$}

\footnote[0]{${}^{*}$The Project-sponsored by NSFC (No. 11571071, 11426036), Innovation Program of Shanghai Municipal Education
Commission (No. 12ZZ063) and Natural Science Foundation of Anhui
Province (No.1408085QA10)}

\author[X. Sun, L. Yan and Q. Zhang]{Xichao Sun${}^{\dag,\natural}$, Litan Yan${}^{\ddag,\S}$ and Qinghua Zhang${}^{\ddag}$}

\footnote[0]{${}^{\natural}$sunxichao626@126.com, ${}^{\S}$litan-yan@hotmail.com (Corresponding Author)}

\date{}

\keywords{Weighted fractional Brownian motion, local time, Malliavin calculus, quadratic covariation, It\^{o} formula}

\subjclass[2000]{60G15, 60H05, 60G17}

\maketitle

\begin{center}
{\footnotesize {\it ${}^{\dag}$Department of Mathematics and Physics, Bengbu University\\
1866 Caoshan Rd., Bengbu 233030, P.R. China\\
${}^{\ddag}$Department of Mathematics, College of Science, Donghua University\\
2999 North Renmin Rd. Songjiang, Shanghai 201620, P.R. China}}
\end{center}

\maketitle

%%%%%%%%%%%%%%%%%%%%%

\begin{abstract}
Let $B^{a,b}$ be a weighted fractional Brownian motion with indices $a,b$ satisfying $a>-1,-1<b<0,|b|<1+a$. In this paper, motivated by the asymptotic property
$$
E[(B^{a,b}_{s+\varepsilon}-B^{a,b}_s)^2] =O(\varepsilon^{1+b})\not\sim \varepsilon^{1+a+b}=E[(B^{a,b}_{\varepsilon})^2]\qquad (\varepsilon\to 0)
$$
for all $s>0$, we consider the generalized quadratic covariation $\bigl[f(B^{a,b}),B^{a,b}\bigr]^{(a,b)}$ defined by
$$
\bigl[f(B^{a,b}),B^{a,b}\bigr]^{(a,b)}_t=\lim_{\varepsilon\downarrow
0}\frac{1+a+b}{\varepsilon^{1+b}}\int_\varepsilon^{t+\varepsilon}
\left\{f(B^{a,b}_{s+\varepsilon})
-f(B^{a,b}_s)\right\}(B^{a,b}_{s+\varepsilon}-B^{a,b}_s)s^{b}ds,
$$
provided the limit exists uniformly in probability. We construct a Banach space ${\mathscr H}$ of measurable functions such that the generalized quadratic covariation exists in $L^2(\Omega)$ and the generalized Bouleau-Yor identity
$$
[f(B^{a,b}),B^{a,b}]^{(a,b)}_t=-\frac1{(1+b){\mathbb B}(a+1,b+1)} \int_{\mathbb R}f(x){\mathscr L}^{a,b}(dx,t)
$$
holds for all $f\in {\mathscr H}$, where ${\mathscr L}^{a,b}(x,t)=\int_0^t\delta(B^{a,b}_s-x)ds^{1+a+b}$ is the weighted local time of $B^{a,b}$ and ${\mathbb B}(\cdot,\cdot)$ is the Beta function.
\end{abstract}

%%\tableofcontents
%%%%%%%%%%%%%%%%%%%%%%%%%%%%%%%%%%%%%%%%%%%%%%%%%%%%%%%%%%%%%%

\section{Introduction}
Long/short range dependence (or long/short memory) stochastic processes with self-similarity have been intensively used as models for different physical phenomena. These properties appeared in empirical studies in areas like hydrology and geophysics; and they appeared to play an important role in network traffic analysis, economics and telecommunications. As a consequence, some efficient mathematical models based on long/short range dependence processes with self-similarity have been proposed in these directions. We refer to the monographs of self-similar processes by Embrechts-Maejima~\cite{Embrechts-Maejima}, Sheluhin et al~\cite{Sheluhin et al.},  Samorodnitsky~\cite{Samorodnitsky},  Samorodnitsky-Taqqu~\cite{Samorodnitsky-Taqqu}, Taqqu~\cite{Taqqu3} and Tudor~\cite{Tudor}.

The fractional Brownian motion is a simple stochastic process with long/short range dependence and self-similarity which is a suitable generalization of standard Brownian motion. Some surveys and complete literatures on fractional Brownian motion could be found in Biagini {\it et al}~\cite{BHOZ}, Gradinaru et al.~\cite{Grad1}, Hu~\cite{Hu2}, Mishura~\cite{Mishura2}, Nualart~\cite{Nualart1}. On the other hand, many authors have proposed to use more general self-similar Gaussian processes and random fields as stochastic models. Such applications have raised many interesting theoretical questions about self-similar Gaussian
processes and fields in general. Therefore, some generalizations of
the fBm has been introduced. However, contrast to the extensive
studies on fractional Brownian motion, there has been little
systematic investigation on other self-similar Gaussian processes.
The main reason for this is the complexity of dependence structures
for self-similar Gaussian processes which do not have stationary
increments. Thus, it seems interesting to study some extensions of fractional Brownian motion such as bi-fractional Brownian motion and the weighted fractional Brownian motion.

In this paper we consider the weighted fractional Brownian motion (weighted-fBm). Recall that the so-called weighted-fBm $B^{a,b}$ with parameters $a$ and $b$ is a zero mean Gaussian process with long/short-range dependence and self-similarity. It admits the relatively simple covariance as follows
$$
E\left[B^{a,b}_tB^{a,b}_s\right]=\frac{1}{2{\mathbb
B}(a+1,b+1)}\int_0^{s\wedge t}u^a((t-u)^b+(s-u)^b)du
$$
where ${\mathbb B}(\cdot,\cdot)$ is the beta function and
$a>-1,|b|<1,|b|<a+1$. Clearly, if $a=0$, the process coincides with
the standard fractional Brownian motion with Hurst parameter
$H=\frac{b+1}2$, and it admits the explicit significance. We have
(see, Lemma~\ref{lem3.0} in Section~\ref{sec3}, see also Bojdecki
{\em et al}~\cite{Bojdecki1})
\begin{equation}\label{sec1-eq1-1}
c_{a,b}(t\vee s)^a|t-s|^{b+1}\leq
E\left[\left(B^{a,b}_t-B^{a,b}_s\right)^2\right]\leq C_{a,b} (t\vee
s)^a|t-s|^{b+1}
\end{equation}
for $s,t\geq 0$. Thus, Kolmogorov's continuity criterion implies
that weighted-fractional Brownian motion is $\gamma$-H\"{o}lder
continuous for any $\gamma<\frac{1+b}2$, where $\frac{1+b}2$ is
called the H\"older continuous index. The process $B^{a,b}$ is $\frac12(a+b+1)$-self similar and its increments are not stationary. It is important to note that the following fact:
\begin{itemize}
\item The H\"older continuous index $\frac12(1+b)$ is not equal either to the its self-similar index nor the order of the infinitesimal $\sqrt{E[(X_t)^2]}\to 0$ as $t\downarrow 0$, provided $a\neq 0$.
\end{itemize}
However, the three indexes are coincident for many famous self-similar Gaussian processes such as fractional Brownian motion,
sub-fractional Brownian motion and bi-fractional Brownian motion.
That is causing trouble for the research, and it is also our a motivation to study the weighted-fBm. Before making the decision to study the weighted-fBm we first try to investigate in Yan et al.~\cite{Yan3} some path properties including strong local nondeterminism, Chung's law of the iterated logarithm and the smoothness of the collision local time. In particular, we showed that it is strongly locally $\phi$-nondeterministic with $\phi(r)=r^{1+b}$. In general, the function $\phi$ depends on the self-similar index of the process, but the fact is, for the weighted-fBm, $\phi(r)=r^{1+b}$ is independent of parameter $a$, which enhances further our interesting to study the weighted-fBm.

The weighted-fBm appeared in Bojdecki {\em et al}~\cite{Bojdecki1} in a limit of occupation time fluctuations of a system of independent particles moving in ${\mathbb R}^d$ according a symmetric $\alpha$-stable L\'evy process, $0<\alpha\leq 2$, started from an inhomogeneous
Poisson configuration with intensity measure
$$
\frac{dx}{1+|x|^\gamma}
$$
and $0<\gamma\leq d=1<\alpha$, $a=-\gamma/\alpha$, $b=1-1/\alpha$,
the ranges of values of $a$ and $b$ being $-1< a < 0$ and $0 < b\leq
1+a$. The process also appears in Bojdecki {\em et
al}~\cite{Bojdecki2} in a high-density limit of occupation time
fluctuations of the above mentioned particle system, where the
initial Poisson configuration has finite intensity measure, with
$d=1<\alpha$, $a=-1/\alpha$, $b=1-1/\alpha$. Moreover, the
definition of the weighted-fBm $B^{a,b}$ was first introduced by
Bojdecki {\em et al}~\cite{Bojdecki0}, and it is neither a
semimartingale nor a Markov process if $b\neq 0$, so many of the
powerful techniques from stochastic analysis are not available when
dealing with $B^{a,b}$. There has been little systematic investigation on weighted-fBm since it it has been introduced by Bojdecki {\em et al}~\cite{Bojdecki0}.

In this paper, we consider the the {\it generalized quadratic covariation} when $b<0$, and it is important to note that a large class of Gaussian processes with similar characteristics as weighted-fBm could be handled in uniform approach used here. Clearly, by the estimates~\eqref{sec1-eq1-1}, we have
$$
E\bigl[(B^{a,b}_{s+\varepsilon}-B^{a,b}_s)^2\bigr]=O( s^a\varepsilon^{1+b})\qquad (\varepsilon\to 0)
$$
for all $s>0$, which implies that
$$
\lim_{\varepsilon\downarrow 0}\frac{1}{\varepsilon} \int_{t_0}^{t}(B_{s+\varepsilon}-B_s)^2ds=
\begin{cases}
0, & \text{ {if $0<b<1$}}\\
+\infty,& \text{ {if $-1<b<0$}}
\end{cases}
$$
for all $t\geq t_0>0$, where the limit is uniformly in probability. Additional results on the quadratic variation can be found in Russo-Vallois~\cite{Russo2}. Thus, we need a substitute tool of the quadratic variation for $b\neq 0$. Inspired by~\eqref{sec1-eq1-1}, the fact
$$
E[(B^{a,b}_{s+\varepsilon}-B^{a,b}_s)^2] =O(\varepsilon^{1+b})\not\sim \varepsilon^{1+a+b}=E[(B^{a,b}_{\varepsilon})^2]\qquad (\varepsilon\to 0)
$$
for all $s>0$ and Cauchy's principal value, one can naturally give the following definition.

\begin{definition}
Let $a>-1,|b|<1,|b|<a+1$ and let the integral

$$
J_\varepsilon(f,t):=\frac{1+a+b}{\varepsilon^{1+b}} \int_\varepsilon^{t+\varepsilon}\left\{ f(B^{a,b}_{s+\varepsilon})
-f(B^{a,b}_s)\right\}(B^{a,b}_{s+\varepsilon}-B^{a,b}_s)s^{b}ds
$$
exists for all $\varepsilon>0$ and all Borel functions $f$. The limit
$$
[f(B^{a,b}),B^{a,b}]^{(a,b)}_t:=\lim_{\varepsilon\downarrow 0}J_\varepsilon(f,t)
$$
is called the {\it generalized quadratic covariation} of $f(B^{a,b})$ and $B^{a,b}$, provided the limit exists uniformly in probability.
\end{definition}

\begin{remark}\label{rem}
{\rm

It is important to note that the above definition is available for a large class of Gaussian processes with similar characteristics as weighted-fBm. Let now $X$ be a self-similar Gaussian process with H\"older continuous paths of order $\alpha\in (0,1)$. We then can define the generalized quadratic covariation $[f(X),X]^{(a,b)}$ as follows
$$
[f(X),X]^{(a,b)}_t=\lim_{\varepsilon\downarrow
0}\frac{2\alpha}{\varepsilon^{2\alpha}}\int_\varepsilon^{t+\varepsilon}
\left\{f(X_{s+\varepsilon})
-f(X_s)\right\}(X_{s+\varepsilon}-X_s)s^{2\alpha-1}ds
$$
for any Borel functions $f$, provided the limit exists uniformly in probability. When $0<\alpha<\frac12$, we can get some similar results for the process $X$ to weighted-fBm with $-1<b<0$.

}
\end{remark}

We shall see in Section~\ref{sec6} that
$$
\bigl[f(B^{a,b}),B^{a,b}\bigr]^{(a,b)}_t =\kappa_{a,b}\int_0^tf'(B^{a,b}_s)ds^{1+a+b},\qquad t\geq 0
$$
for all $f\in C^1({\mathbb R})$ and all $b\in (-1,1)$, where
$$
\kappa_{a,b}=\frac1{(1+b){\mathbb B}(a+1,b+1)}.
$$
In the present paper we prove the existence of the generalized quadratic covariation for $-1<b<0$, our start point is to consider the decomposition
\begin{equation}\label{sec1-eq1-2}
\begin{split}
\frac{1}{\varepsilon^{1+b}}\int_\varepsilon^{t+\varepsilon} &\left\{f(B^{a,b}_{
s+\varepsilon})-f(B^{a,b}_s)\right\}(B^{a,b}_{s+\varepsilon} -B^{a,b}_s)s^{b}ds\\
&=\frac{1}{\varepsilon^{1+b}}\int_\varepsilon^{t+\varepsilon} f(B^{a,b}_{
s+\varepsilon})(B^{a,b}_{s+\varepsilon}-B^{a,b}_s)s^bds\\ &\qquad\qquad-\frac{1}{\varepsilon^{1+b}}
\int_\varepsilon^{t+\varepsilon} f(B^{a,b}_s)(B^{a,b}_{s+\varepsilon}-B^{a,b}_s)s^bds
\end{split}
\end{equation}
for all $-1<b<0$. It is important to note that the above decomposition is unavailable for $0<b<1$. For example, we have
$$
\frac{1}{\varepsilon^{1+b}}\int_\varepsilon^{t+\varepsilon} E\left(B^{a,b}_{s} (B^{a,b}_{s+\varepsilon}-B^{a,b}_s)\right)s^bds\longrightarrow \infty
$$
for all $b>0$ and $t>0$, as $\varepsilon$ tends to zero, because
\begin{equation}\label{sec1-eq1-3}
E\left(B^{a,b}_{s} (B^{a,b}_{s+\varepsilon}-B^{a,b}_s)\right) \sim s^{a+b}\varepsilon\qquad (\varepsilon\to 0).
\end{equation}
The above asymptotic property follows from
\begin{align*}
E[B^{a,b}_s(B^{a,b}_{s+\varepsilon}&-B^{a,b}_s)]=\frac1{2{\mathbb B}(1+a,1+b)}\int_0^su^a[({s+\varepsilon}-u)^b+(s-u)^b]du-s^{1+a+b}\\
&=\frac1{2{\mathbb B}(1+a,1+b)}\int_0^su^a\left[({s+\varepsilon}-u)^b-(s-u)^b\right]du\\
&=\frac1{2{\mathbb B}(1+a,1+b)}\int_0^su^a({s+\varepsilon}-u)^b\left[1- (\frac{s-u}{{s+\varepsilon}-u})^b\right]du\\
&\sim s^{a+b}\varepsilon
\end{align*}
for all $b>0$ by the fact $1-x^b\sim 1-x$ as $x\to 1$. Thus, the method used here is different  from that need to handle the case $0<b<1$.

This paper is organized as follows. In Section~\ref{sec2} we present some preliminaries for weighted-fBm and Malliavin calculus. In Section~\ref{sec3}, we establish some technical estimates associated with weighted-fBm with $-1<b<0$. In Section~\ref{sec6}, as an example, when $f\in C^1({\mathbb R})$ we show that the generalized quadratic covariation $\bigl[f(B^{a,b}),B^{a,b}\bigr]^{(a,b)}$ exists in $L^2$ for all $a,b$ and
$$
\bigl[f(B^{a,b}),B^{a,b}\bigr]^{(a,b)}_t=\kappa_{a,b}\int_0^t f'(B^{a,b}_s)ds^{1+a+b}.
$$
In particular, we have
$$
\bigl[B^{a,b},B^{a,b}\bigr]^{(a,b)}_t=\kappa_{a,b}t^{1+a+b}.
$$
In Section~\ref{sec5}, in more general cases we consider the existence of generalized quadratic covariation for $-1<b<0$. By estimating the two terms of the right hand side in the decomposition~\eqref{sec1-eq1-2}, respectively, we construct a Banach space ${\mathscr H}$ of measurable functions $f$ on
${\mathbb R}$ such that $\|f\|_{{\mathscr H}}<\infty$, where
\begin{align*}
(\|f\|_{{\mathscr H}})^2:=\int_0^{T+1}\int_{\mathbb
R}|f(x)|^2e^{-\frac{x^2}{2s^{1+a+b}
}}\frac{dxds}{\sqrt{2\pi}s^{(1-a-b)/2}}.
\end{align*}
We show that the {\it generalized quadratic covariation} $\bigl[f(B^{a,b}),B^{a,b}\bigr]^{(a,b)}$ exists for all $f\in {\mathscr H}$ and
$$
E\left|\bigl[f(B^{a,b}),B^{a,b}\bigr]^{(a,b)}_t\right|^2\leq C_{a,b,T}\|f\|_{\mathscr H}^2,\qquad 0\leq t\leq T.
$$
In Section~\ref{sec7}, for $-1<b<0$ we consider the integral
\begin{equation}\label{sec1-eq1-4}
\int_{\mathbb R}f(x){\mathscr L}^{a,b}(dx,t)
\end{equation}
for $f\in {\mathscr H}$, where ${\mathscr L}^{a,b}(x,t)$ denotes the weighted local time defined by
$$
{\mathscr L}^{a,b}(x,t)=(1+a+b)\int_0^t\delta(B^{a,b}_s-x)s^{a+b}ds.
$$
In order to study the integral we obtain the following It\^o formula :
$$
F(B^H)=F(0)+\int_0^tf(B^{a,b}_s)dB^{a,b}_s+\frac12(\kappa_{a,b})^{-1}
[f(B^{a,b}),B^{a,b}]^{(a,b)}_t,
$$
where $F$ is an absolutely continuous function with $F'=f\in {\mathscr H}$, and show that the {\em generalized Bouleau-Yor identity}
$$
[f(B^{a,b}),B^{a,b}]^{(a,b)}_t=-\kappa_{a,b}\int_{\mathbb R}f(x){\mathscr L}^{a,b}(dx,t)
$$
holds for all $f\in {\mathscr H}$. As a corollary we get the Tanaka formula
$$
|B^{a,b}_t-x|=|-x|+\int_0^t{\rm sign}(B^{a,b}_s-x)dB^{a,b}_s+{\mathscr L}^{a,b}(x,t)
$$
for $-1<b<0$.

\section{The weighted fractional Brownian motion}\label{sec2}
%%%%%%%%%%%%%%%%%%%%%%%%%%%%%%%%%%%%%%%%%%%%%%%%%%%%%%%%%%%%%
%%%%%%%%%%%%%%%%%%%%%%%%%%%%%%%%%%%%%%%%%%%%%%%%%%%%%%%%%%%%%
Let $B^{a,b}$ be a weighted-fBm with parameters $a,b$
($a>-1,|b|<1,|b|<a+1$), defined on the complete probability space
$(\Omega,\mathcal{F},P)$. As we pointed out before, the weighted-fBm
$B^{a,b}=\left\{B^{a,b}_t,0\leq t\leq T \right\}$ with indices $a$
and $b$ is a mean zero Gaussian processes such that $B_0^{a,b}=0$ and
\begin{equation}\label{sec2-eq2.1}
E\left[B^{a,b}_tB^{a,b}_s\right]=\frac{1}{2{\mathbb
B}(a+1,b+1)}\int_0^{s\wedge t}u^a((t-u)^b+(s-u)^b)du
\end{equation}
for $s,t\geq 0$. It is known that the function $(t,s)\mapsto
R^{a,b}(t,s)$ is positive-definite if and only if $a$ and $b$
satisfy the conditions
\begin{equation}\label{sec2-eq2.2}
a>-1,\;|b|<1,\;|b|<a+1,
\end{equation}
and the following statements hold (see Bojdecki~\cite{Bojdecki0}):
\begin{itemize}
\item $B^{a,b}$ is $\frac12(a+b+1)$-self similar;
\item $B^{a,b}$ has independent increments for $b=0$;
\item $B^{a,b}$ is neither a semimartingale nor a Markov process if $b\neq 0$;
\item If $b>0$, then $B^{a,b}$ is long-range dependence;
\item If $b<0$, then $B^{a,b}$ is short-range dependence.
\end{itemize}
Thus, throughout this paper we let $b\neq 0$ for simplicity.

As a Gaussian process, it is possible to construct a stochastic calculus of variations with respect to $B^{a,b}$. We refer to Al\'os {\it et al}~\cite{Nua1} and Nualart~\cite{Nualart1} for the complete descriptions of stochastic calculus with respect to Gaussian processes. Here we recall only the basic elements of this theory. Throughout this paper we assume that~\eqref{sec2-eq2.2} holds.

Let ${\mathcal H}_{a,b}$ be the completion of the linear space ${\mathcal
E}$ generated by the indicator functions $1_{[0,t]}, t\in [0,T]$
with respect to the inner product
$$
\langle 1_{[0,s]},1_{[0,t]} \rangle_{{\mathcal H}_{a,b}}=R^{a,b}(s,t)=\frac{1}{2{\mathbb
B}(a+1,b+1)}\int_0^{s\wedge t}u^a((t-u)^b+(s-u)^b)du.
$$
The application $\varphi\in {\mathcal E}\to B(\varphi)$ is an
isometry from ${\mathcal E}$ to the Gaussian space generated by $B^{a,b}$ and it can be extended to ${\mathcal H}_{a,b}$.

\begin{remark}
{\rm

For $b>0$ we can characterize ${\mathcal H}_{a,b}$ as
$$
{\mathcal H}_{a,b}=\{f:[0,T]\to {\mathbb R}\;|\;\|f\|_{a,b}<\infty\},
$$
where
$$
\|f\|^2_{a,b}=\int_0^T\int_0^Tf(t)f(s)\frac{\partial^2}{\partial s\partial t}R^{a,b}(s,t)dsdt.
$$
Clearly, we can write its covariance as
\begin{align*}
\phi_{a,b}(s,t):=\frac{\partial^2}{\partial t\partial s}R^{a,b}(t,s) &=\frac{b}{2{\mathbb
B}(a+1,b+1)}(t\wedge s)^a|t-s|^{b-1}
\end{align*}
for $b>0$. Thus, $R^{a,b}$ is the distribution function of an absolutely continuous positive measure with density $\frac{b}{2{\mathbb
B}(a+1,b+1)}(t\wedge s)^a|t-s|^{b-1}$ which belongs of course to
$L^1([0,T ]^2)$.

}
\end{remark}

Let us denote by ${\mathcal S}^{a,b}$ the set of smooth functionals of the form
$$
F=f(B(\varphi_1),B(\varphi_2),\ldots,B(\varphi_n)),
$$
where $f\in C^{\infty}_b({\mathbb R}^n)$ ($f$ and all its derivatives are bounded) and $\varphi_i\in {\mathcal H}_{a,b}$. The {\em Malliavin derivative} $D^{a,b}$ of a functional $F$ as above is given by
$$
D^{a,b}F=\sum_{j=1}^n\frac{\partial f}{\partial
x_j}(B(\varphi_1),B(\varphi_2), \ldots,B(\varphi_n))\varphi_j.
$$
The derivative operator $D^{a,b}$ is then a closable operator from $L^2(\Omega)$ into $L^2(\Omega;{\mathcal H}_{a,b})$. We denote by ${\mathbb D}^{1,2}$ the closure of ${\mathcal S}_{a,b}$ with respect to the norm
$$
\|F\|_{1,2}:=\sqrt{E|F|^2+E\|D^{a,b}F\|^2_{a,b}}.
$$

The {\it divergence integral} $\delta^{a,b}$ is the adjoint of derivative operator $D^{a,b}$. That is, we say that a random variable $u$ in $L^2(\Omega;{\mathcal H}_{a,b})$ belongs to the domain of the divergence operator $\delta^{a,b}$, denoted by ${\rm {Dom}}(\delta^{a,b})$, if
$$
E\left|\langle D^{a,b}F,u\rangle_{{\mathcal H}_{a,b}}\right|\leq
c\|F\|_{L^2(\Omega)}
$$
for every $F\in {\mathbb D}^{1,2}$, where $c$ is a constant depending only on $u$. In this case $\delta^{a,b}(u)$ is defined by the duality relationship
\begin{equation}
E\left[F\delta^{a,b}(u)\right]=E\langle D^{a,b}F,u\rangle_{{\mathcal H}_{a,b}}
\end{equation}
for any $F\in {\mathbb D}^{1,2}$, we have ${\mathbb D}^{1,2}\subset
{\rm {Dom}}(\delta^{a,b})$. We will denote
$$
\delta^{a,b}(u)=\int_0^Tu_sdB^{a,b}_s
$$
for an adapted process $u$, and it is called Skorohod integral. We have the following It\^o formula.

\begin{theorem}[Al\'os {\it et al}~\cite{Nua1}]\label{theorem-Ito}
Let $f\in C^{2}({\mathbb R})$ such that
\begin{equation}\label{sec2-Ito-con1}
\max\left\{|f(x)|,|f'(x)|,|f''(x)|\right\}\leq \kappa e^{\beta x^2},
\end{equation}
where $\kappa$ and $\beta$ are positive constants with
$\beta<\frac14T^{-(1+a+b)}$. Then we have
\begin{align*}
f(B^{a,b}_t)=f(0)&+\int_0^t\frac{d}{dx}f(B^{a,b}_s)dB^{a,b}_s
+\frac12(1+a+b)\int_0^t\frac{d^2}{dx^2}f(B^{a,b}_s)s^{a+b}ds
\end{align*}
for all $t\in [0,T]$.
\end{theorem}

%%%%%%%%%%%%%%%%%%%%%%%%%%%%%%%%%%%%%%%%%%%%%%%%%%%%%%%%%%%%%%%%%
%%%%%%%%%%%%%%%%%%%%%%%%%%%%%%%%%%%%%%%%%%%%%%%%%%%%%%%%%%%%%%%%%

\section{Some basic estimates}\label{sec3}

For simplicity throughout this paper we let $C$ stand for a positive constant depending only on the subscripts and its value may be different in different appearance, and this assumption is also adaptable to $c$. Moreover, the notation $F\asymp G$ means that there are positive constants $c_1$ and $c_2$
so that
$$
c_1G(x)\leq F(x)\leq c_2G(x)
$$
in the common domain of definition for $F$ and $G$. For $x,y\in
\mathbb{R}$, $x\wedge y:=\min\{x,y\}$ and $x\vee y:=\max\{x,y\}$.
\begin{lemma}\label{lem3.0}
Let $a>-1,\;|b|<1,\;|b|<a+1$. We then have
\begin{equation}\label{sec3-eq3.1}
Q(t,s):=E\left[\left(B^{a,b}_t-B^{a,b}_s\right)^2\right]\asymp (t\vee
s)^a|t-s|^{1+b}
\end{equation}
for $s,t\geq 0$. In particular, we have
\begin{equation}\label{sec3-eq3.2}
E\left[\left(B^{a,b}_t-B^{a,b}_s\right)^2\right]\leq
C_{a,b}|t-s|^{1+a+b}
\end{equation}
for $a\leq 0$.
\end{lemma}
The estimates~\eqref{sec3-eq3.1} are first considered by
Bojdecki~\cite{Bojdecki0}. The present form is a slight
modification given by Yan et al.~\cite{Yan3}. Thus, Kolmogorov's continuity criterion and the Gaussian property of the process imply that
weighted-fractional Brownian motion is $\gamma$-H\"{o}lder
continuous for any $\gamma<\frac{1+b}2$, where $\frac{1+b}2$ is
called the {\em H\"older continuity index}. It is important to note
that the H\"older continuity index $\frac{1+b}2$ is not equal to the
order $\frac{1+a+b}2$ of the infinitesimal $\sqrt{E[(X_t)^2]}\to 0$
as $t\downarrow 0$ unless $a=0$. However, the H\"{o}lder continuity
index of many popular self-similar Gaussian processes equals to the
order of the infinitesimal such as fractional Brownian motion,
sub-fractional Brownian motion and bi-fractional Brownian motion.

\begin{lemma}\label{lem3.1}
Let $a>-1,\;|b|<1,\;|b|<a+1$. We then have
\begin{equation}\label{sec3-eq3.3}
t^{1+a+b}s^{1+a+b}-\mu^2\asymp (ts)^{a}(t\wedge s)^{1+b}|t-s|^{1+b}
\end{equation}
for all $s,t>0$, where $\mu=E(B^{a,b}_tB^{a,b}_s)$.
\end{lemma}
\begin{proof}
Without loss of generality we may assume that $t>s>0$. Then
\begin{align*}
t^{1+a+b}&s^{1+a+b}-\mu^2\\
&=t^{1+a+b}s^{1+a+b}-\frac1{4{\mathbb B}^2(1+a,1+b)}\left(\int_0^su^a\left((t-u)^b+(s-u)^b\right)du\right)^2\\
&\equiv t^{2(1+a+b)}G(x)
\end{align*}
with $x=\frac{s}{t}$, where
$$
G(x)=x^{1+a+b}-\frac1{4{\mathbb B}^2(1+a,1+b)}\left(\int_0^xu^a(1-u)^bdu+{\mathbb B}(1+a,1+b)x^{1+a+b}\right)^2
$$
with $x\in [0,1]$. Noting that
$$
t^{1+a+b}s^{1+a+b}-\mu^2\geq 0
$$
and
$$
t^{1+a+b}s^{1+a+b}-\mu^2=0\qquad\Longleftrightarrow\qquad s=t\quad {\text {or }}\quad s=0
$$
for all $t\geq s>0$, we see that
$G(x)\geq 0$ and
$$
G(x)=0\qquad \Longleftrightarrow \qquad x=0\quad{\text {or }}\quad x=1.
$$
Decompose $G(x)$ as follows
\begin{align*}
G(x)&=\frac1{4(K_{a,b})^2}\left(2K_{a,b} x^{\frac{1+a+b}2}-\int_0^xu^a(1-u)^bdu-K_{a,b}x^{1+a+b}\right)\\
&\qquad\quad\cdot\left(2K_{a,b} x^{\frac{1+a+b}2}+\int_0^xu^a(1-u)^bdu+K_{a,b}x^{1+a+b}\right)\\
&\equiv \frac1{4(K_{a,b})^2}G_1(x)G_2(x)
\end{align*}
for all $x\in [0,1]$, where $K_{a,b}={\mathbb B}(1+a,1+b)$.

Obviously, we have
$$
K_{a,b} x^{\frac{1+a+b}2}\leq G_2(x)\leq C_{a,b}x^{\frac{1+a+b}2}
$$
for all $x\in [0,1]$. In fact, the left inequality is clear, and the right inequality follows from
the fact
$$
\int_0^xv^a(1-v)^bdv\leq \int_0^xv^a\left((x-v)^b+1\right)dv\leq K_{a,b}x^{1+a+b}+\frac1{1+a}x^{1+a}
$$
for all $x\in [0,1]$. On the other hand, we also have
\begin{align*}
\lim_{x\uparrow 1}\frac{G_1(x)}{x^{\frac{1+a+b}2}(1-x)^{1+b}}=\frac1{1+b},\qquad \lim_{x\downarrow 0}\frac{G_1(x)}{x^{\frac{1+a+b}2}(1-x)^{1+b}}=2{\mathbb B}(1+a,1+b),
\end{align*}
which deduces
$$
G_1(x)\asymp x^{\frac{1+a+b}2}(1-x)^{1+b}
$$
by the continuity. Thus, we have showed that
$$
G(x)=G_1(x)G_2(x)\asymp x^{1+a+b}(1-x)^{1+b}
$$
and the lemma follows.
\end{proof}
\begin{lemma}\label{lem3.3}
Let $t>s>t'>s'>0$ and let $-1<b<1$, $a>-1$, $|b|<1+a$. We then have
\begin{equation}\label{sec3-eq3.6}
\begin{split}
|E(B^{a,b}_{t}-&B^{a,b}_s)(B^{a,b}_{t'}-B^{a,b}_{s'})|
\\
&\leq C_{a,b,\alpha}\left((s')^a\vee s^a\right)^\alpha(tt')^{\frac12a(1-\alpha)} \frac{[(t-s)(t'-s')]^{\alpha+\frac12(1-\alpha)(1+b)} }{(t-t')^{(1-b)\alpha}}
\end{split}
\end{equation}
for all $\alpha\in [0,1]$.
\end{lemma}
\begin{proof}
Let $b<0$. Denote
\begin{align*}
\mu(t,s,t',s'):&=E(B^{a,b}_{t}-B^{a,b}_s)(B^{a,b}_{t'}-B^{a,b}_{s'})\\
&=\frac1{2{\mathbb B}(1+a,1+b)}\int_{s'}^{t'}u^a\left[(s-u)^b-(t-u)^b\right]du.
\end{align*}
It follows from the fact
\begin{equation}\label{sec3-eq3.7=0}
\begin{split}
y^\gamma-x^\gamma&=y^\gamma\left(1-(\frac{x}{y})^\gamma\right)\leq C_{\gamma}y^\gamma\left(1-\frac{x}{y}\right)\\
&\leq C_{\gamma,\beta}y^\gamma\left(1-\frac{x}{y}\right)^\beta
\leq C_{\gamma,\beta}y^{\gamma-\beta}(y-x)^\beta
\end{split}
\end{equation}
for $y>x>0$, $\gamma\geq 0$, $0\leq \beta\leq 1$ and the inequality
\begin{align*}
t-u=(t-t')+(t'-u)\geq (t-t')^{1-\nu}(t'-u)^\nu\qquad (0<u<t'<t)
\end{align*}
for all $0\leq \nu\leq 1$ that
\begin{equation}\label{sec3-eq3.8==}
\begin{split}
|\mu(t,s,&t',s')|=\frac1{2{\mathbb B}(1+a,1+b)}\int_{s'}^{t'}u^a\left[(s-u)^b -(t-u)^b\right]du\\
&\leq C_{a,b}\int_{s'}^{t'}u^a\frac{(t-u)^{-b} -(s-u)^{-b}}{
(s-u)^{-b}(t-u)^{-b}}du\leq C_{a,b,\beta}\int_{s'}^{t'}u^a\frac{(t-s)^\beta}{
(s-u)^{-b}(t-u)^\beta}du\\
&\leq C_{a,b}\frac{(t-s)^\beta}{(t-t')^{\beta(1-\nu)}} \int_{s'}^{t'}\frac{u^a}{(t'-u)^{-b+\beta\nu}}du\\
&\leq C_{a,b}\left[(s')^a\vee s^a\right]\frac{(t-s)^\beta(t'-s')^{1+b-\beta\nu}}{(t-t')^{\beta(1-\nu)}}
\end{split}
\end{equation}
for all $0\leq \beta\leq 1$ and $0\leq \nu\beta\leq 1+b$.

On the other hand, noting that
\begin{align*}
|E[(B^{a,b}_t-B^{a,b}_s)(B^{a,b}_{t'}-B^{a,b}_{s'})]|^2&\leq E\left[(B^{a,b}_t-B^{a,b}_s)^2\right]E\left[(B^{a,b}_{t'} -B^{a,b}_{s'})^2\right]\\
&\leq C_{a,b}(tt')^a(t-s)^{1+b}(t'-s')^{1+b},
\end{align*}
we see that
\begin{align*}
&\frac{|E[(B^{a,b}_t-B^{a,b}_s)(B^{a,b}_{t'}-B^{a,b}_{s'})]|}{ {\sqrt{C_{a,b}(tt')^a(t-s)^{1+b}(t'-s')^{1+b}}}}\leq
\left(\frac{|E[(B^{a,b}_t-B^{a,b}_s)(B^{a,b}_{t'}-B^{a,b}_{s'})]|}{ {\sqrt{C_{a,b}(tt')^a(t-s)^{1+b}(t'-s')^{1+b}}}}\right)^\alpha
\end{align*}
for all $\alpha\in [0,1]$. Combining this with~\eqref{sec3-eq3.8==} (taking $\beta=1+b$ and $\nu=0$), we get
\begin{align*}
|E[(B^{a,b}_t-&B^{a,b}_s)(B^{a,b}_{t'}-B^{a,b}_{s'})]|\\
&\leq
C_{a,b,\alpha}\left((s')^a\vee s^a\right)^\alpha(tt')^{\frac12a(1-\alpha)} \frac{[(t-s)(t'-s')]^{\alpha+\frac12(1-\alpha)(1+b)} }{(t-t')^{(1+b)\alpha}}
\end{align*}
and the lemma follows for all $\alpha\in [0,1]$. Similarly, we can show that the lemma holds for $b>0$.
\end{proof}

\begin{lemma}\label{lem3.4}
For $a>-1$, $-1<b<0$ and $|b|<1+a$ we have
\begin{align}\label{lem3.4-eq1}
&E[B^{a,b}_t(B^{a,b}_s-B^{a,b}_r)]\leq C_{a,b}(s-r)^{1+b}s^a\\   \label{lem3.4-eq2}
&E[B^{a,b}_s(B^{a,b}_t-B^{a,b}_s)]\leq C_{a,b}(t-s)^{1+b}s^a\\   \label{lem3.4-eq3}
&E[B^{a,b}_s(B^{a,b}_s-B^{a,b}_r)]\leq C_{a,b}(s-r)^{1+b}s^a\\   \label{lem3.4-eq4}
&E[B^{a,b}_{s}(B^{a,b}_t-B^{a,b}_r)]\leq C_{a,b}(t-r)^{1+b}s^a\\  \label{lem3.4-eq5}
&E[B^{a,b}_{r}(B^{a,b}_t-B^{a,b}_s)]\leq C_{a,b}(t-s)^{1+b}r^a
\end{align}
for all $t>s>r>0$.
\end{lemma}
\begin{proof}
Let $t>s>r>0$. By~\eqref{sec3-eq3.7=0} we have
\begin{equation}\label{lem3.4-eq6}
s^{1+a+b}-r^{1+a+b}\leq C_{a,b}s^{a+b}(s-r)\leq C_{a,b}(s-r)^{1+b}s^a
\end{equation}
for all $a>-1$, $-1<b<0$ and $0\leq \beta\leq 1$. Notice that
\begin{equation}\label{lem3.1-eq3}
\int_{x}^1r^a(1-r)^bdr \asymp (1-x)^{1+b}
\end{equation}
for all $x\in [0,1]$ by the continuity and the convergence
\begin{align*}
\lim_{x\to 1}\frac{f(x)}{(1-x)^{1+b}}=\frac1{1+b}
\end{align*}
for all $a,b>-1$. We get
\begin{align*}
\int_r^su^a(t-u)^bdu&\leq \int_r^su^a(s-u)^bdu=s^{1+a+b}\int_{r/s}^1v^a(1-v)^bdv\\
&\leq C_{a,b}s^{1+a+b}(1-\frac{r}s)^{1+b} =C_{a,b}(s-r)^{1+b}s^a
\end{align*}
for $-1<b<0,a>-1$. It follows that
\begin{align*}
E[B^{a,b}_t&(B^{a,b}_s-B^{a,b}_r)]=\frac1{2{\mathbb B}(1+a,1+b)}\left(\int_0^su^a[(t-u)^b+(s-u)^b]du\right.\\
&\hspace{3cm}\left.-
\int_0^ru^a[(t-u)^b+(r-u)^b]du\right)\\
&=\frac1{2{\mathbb B}(1+a,1+b)}\int_r^su^a(t-u)^bdu +\frac12\left(s^{1+a+b}-r^{1+a+b}\right)\\
&\leq C_{a,b}(s-r)^{1+b}s^a.
\end{align*}
This establishes the estimate~\eqref{lem3.4-eq1}. In order to prove the estimate~\eqref{lem3.4-eq2} we have
\begin{equation}\label{lem3.4-eq7}
\int_0^{x}v^a(1-v)^bdv\asymp x^{1+a}
\end{equation}
for all $x\in [0,1]$ by the continuity and the convergence
$$
\lim_{x\downarrow 0}\frac{1}{x^{1+a}}\int_0^xv^a(1-v)^bdv=\frac1{1+a}
$$
for all $a,b>-1$. It follows that
\begin{align*}
|E[B^{a,b}_s(B^{a,b}_t&-B^{a,b}_s)]|=\left|\frac1{2{\mathbb B}(1+a,1+b)}\int_0^su^a[(t-u)^b+(s-u)^b]du-s^{1+a+b}\right|\\
&=\left|\frac1{2{\mathbb B}(1+a,1+b)}\int_0^su^a(t-u)^bdu-\frac12s^{1+a+b}\right|\\
&=\left|\frac{t^{1+a+b}}{2{\mathbb B}(1+a,1+b)}\int_0^{s/t}v^a(1-v)^bdv-\frac12s^{1+a+b}\right|\\
&\leq \frac{1}{2{\mathbb B}(1+a,1+b)}\int_0^{s/t}v^a(1-v)^bdv\left(t^{1+a+b}-s^{1+a+b}\right)\\
&\qquad +s^{1+a+b}\frac{1}{2{\mathbb B}(1+a,1+b)}\left({\mathbb B}(1+a,1+b)-\int_0^{s/t}v^a(1-v)^bdv\right)\\
&\leq C_{a,b} (t-s)^{1+b}s^a,
\end{align*}
which gives~\eqref{lem3.4-eq2}. Similarly, we can prove the estimate~\eqref{lem3.4-eq3}. The estimate~\eqref{lem3.4-eq4} follows from~\eqref{lem3.4-eq2} and~\eqref{lem3.4-eq3}.

Finally, in order to prove~\eqref{lem3.4-eq5} we have
\begin{align*}
|E[B^{a,b}_{r}&(B^{a,b}_t-B^{a,b}_s)]|=\frac1{2{\mathbb B}(1+a,1+b)}\int_0^ru^a[(s-u)^b-(t-u)^b]du\\
&\leq \frac{r^a}{2{\mathbb B}(1+a,1+b)}\int_0^r[(s-u)^b-(t-u)^b]du\\
&\leq \frac{r^a}{2{\mathbb B}(1+a,1+b)}\int_0^s[(s-u)^b-(t-u)^b]du\\
&=C_{a,b}r^a\left(s^{1+b}-t^{1+b}+(t-s)^{1+b}\right)\leq C_{a,b}(t-s)^{1+b}r^a
\end{align*}
for all $a\geq 0$ and $-1<b<0$. Moreover, we also have
\begin{align*}
s^{1+a+b}\Bigl(\int_0^{r/s}& v^a(1-v)^bdv -\int_0^{r/t}v^a(1-v)^bdv\Bigr) =s^{1+a+b}\int_{r/t}^{r/s}v^a(1-v)^bdv\\
&\leq \frac{s^{1+a+b}}{1+b}(\frac{r}t)^a \left(\frac{r}s-\frac{r}{t}\right)^{1+b}\\
&=\frac{s^{1+a+b}}{1+b}(\frac{r}t)^a \frac{(t-s)^{1+b}r^{1+b}}{(ts)^{1+b}}\leq \frac{1}{1+b}r^a (t-s)^{1+b}
\end{align*}
for $-1<a<0$. It follows from~\eqref{lem3.4-eq7} that
\begin{align*}
|E[B^{a,b}_{r}&(B^{a,b}_t-B^{a,b}_s)]|=\frac1{2{\mathbb B}(1+a,1+b)}\int_0^ru^a[(s-u)^b-(t-u)^b]du\\
&=C_{a,b}\left(s^{1+a+b}
\int_0^{r/s}v^a(1-v)^bdv-t^{1+a+b}\int_0^{r/t}v^a(1-v)^bdv\right)\\
&\leq C_{a,b}
s^{1+a+b}
\left(\int_0^{r/s}v^a(1-v)^bdv-\int_0^{r/t}v^a(1-v)^bdv\right)\\
&\qquad +C_{a,b}\int_0^{r/t}v^a(1-v)^bdv
\left(t^{1+a+b}-s^{1+a+b}\right)\\
&\leq C_{a,b} (t-s)^{1+b}r^a+C_{a,b}(\frac{r}{t})^{1+a} \left(t^{1+a+b}-s^{1+a+b}\right)\\
&\leq C_{a,b}(t-s)^{1+b}r^a,
\end{align*}
and the lemma follows.
\end{proof}

Let $\varphi_{t,s}(x,y)$ denote the density function of $(B^{a,b}_t,B^{a,b}_s)$ ($t>s>0$). That is
\begin{equation}\label{sec4-eq4.2}
\varphi_{t,s}(x,y)=\frac1{2\pi\rho_{t,s}}\exp\left\{
-\frac{1}{2\rho^2_{t,s}}\left(s^{1+a+b}x^2-2\mu_{t,s}
xy+t^{1+a+b}y^2\right)\right\},
\end{equation}
where $\mu_{t,s}=E(B^{a,b}_tB^{a,b}_s)$, $E\left[(B^{a,b}_t)^2\right]=t^{1+a+b}$ and
$\rho^2_{t,s}=(ts)^{1+a+b}-\mu_{t,s}^2$.

\begin{lemma}\label{lem3.5}
Let $a>-1,\;|b|<1,\;|b|<a+1$. If $f\in C^1({\mathbb R})$ admits compact support, we then have
\begin{align}\label{lemma3.5-1}
|Ef'(B^{a,b}_{s})f'(B^{a,b}_{r})|&\leq
\left(\frac{(rs)^{\frac12(1+a+b)}}{\rho^2_{s,r}} +\frac{\mu_{s,r}}{\rho^2_{s,r}}\right)\left(E[f^2(B^{a,b}_s)]
+E[f^2(B^{a,b}_r)]\right)\\   \label{lemma3.5-2}
|Ef''(B^{a,b}_{s})f(B^{a,b}_{r})|&\leq
\frac{r^{1+a+b}}{\rho^2_{s,r}}\left(E[f^2(B^{a,b}_s)]
+E[f^2(B^{a,b}_r)]\right)
\end{align}
for all $s,r>0$.
\end{lemma}
\begin{proof}
Elementary computation shows that
\begin{align*}
\int_{\mathbb{R}^2}f^2(y)&(x-\frac{\mu_{t,s}}{r^{1+a+b}}
y)^2\varphi_{s,r}(x,y)dxdy\\
&=\frac{\rho^2_{s,r}}{r^{1+a+b}}\int_{\mathbb{R}}f^2(y)
\frac{1}{\sqrt{2\pi}r^{(1+a+b)/2}}
e^{-\frac{y^2}{2r^{1+a+b}}}dy=\frac{\rho^2_{s,r}}{r^{1+a+b}}
E|f(B^{a,b}_r)|^2,
\end{align*}
which implies that
\begin{align*}
\frac1{\rho^4_{s,r}}\int_{\mathbb{R}^2}|f(x)f(y)
(&s^{1+a+b}y-\mu_{s,r}x)(r^{1+a+b}x-\mu_{s,r}y)|\varphi_{s,r}(x,y)dxdy\\
&\leq \frac{r^{(1+a+b)/2}s^{(1+a+b)/2}}{\rho^2_{s,r}}\left(
E|f(B^{a,b}_s)|^2E|f(B^{a,b}_r)|^2\right)^{1/2}
\end{align*}
for all $s,r>0$. It follows that
\begin{align*}
|E&[f'(B^{a,b}_{s})f'(B^{a,b}_{r})]|=|\int_{\mathbb{R}^2}
f(x)f(y)\frac{\partial^{2}}{\partial x\partial
y}\varphi_{s,r}(x,y)dxdy|\\
&=|\int_{\mathbb{R}^2} f(x)f(y)\left\{\frac1{\rho^4_{s,r}}(s^{1+a+b}y-\mu_{s,r}x)(r^{1+a+b}x -\mu_{s,r}y)+\frac{\mu_{s,r}}{\rho^2_{s,r}}\right\}\varphi_{s,r}(x,y)dxdy|\\
&\leq \left(\frac{(rs)^{(1+a+b)/2}}{\rho^2_{s,r}} +\frac{\mu_{s,r}}{\rho^2_{s,r}}\right)\left(
E|f(B^{a,b}_s)|^2E|f(B^{a,b}_r)|^2\right)^{1/2}.
\end{align*}
Similarly, one can show that the estimate~\eqref{lemma3.5-2} holds.
\end{proof}

%%%%%%%%%%%%%%%%%%%%%%%%%%%%%%%%%%%%%%%%%%%%%%%%%%%%%%%%%%%%%%%%
%%%%%%%%%%%%%%%%%%%%%%%%%%%%%%%%%%%%%%%%%%%%%%%%%%%%%%%%%%%%%%%%
%%%%%%%%%%%%%%%%%%%%%%%%%%%%%%%%%%%%%%%%%%%%%%%%%%%%%%%%%%%%%%%%
%%%%%%%%%%%%%%%%%%%%%%%%%%%%%%%%%%%%%%%%%%%%%%%%%%%%%%%%%%%%%%%%
%%%%%%%%%%%%%%%%%%%%%%%%%%%%%%%%%%%%%%%%%%%%%%%%%%%%%%%%%%%%%%%%

\section{The generalized quadratic covariation, an example} \label{sec6}
In this section, for $|b|<1$ we consider an example of Borel functions such that the generalized quadratic covariation exists. Recall that
$$
J_\varepsilon(f,t)=\frac{1+a+b}{\varepsilon^{1+b}} \int_\varepsilon^{t+\varepsilon}\left\{ f(B^{a,b}_{s+\varepsilon})
-f(B^{a,b}_s)\right\}(B^{a,b}_{s+\varepsilon}-B^{a,b}_s)s^{b}ds
$$
for all $\varepsilon>0$ and all Borel functions $f$, and the generalized quadratic covariation defined as follows
$$
[f(B^{a,b}),B^{a,b}]^{(a,b)}_t=\lim_{\varepsilon\downarrow 0}J_\varepsilon(f,t),
$$
provided the limit exists uniformly in probability (in short, ucp). For ucp-convergence we have the next result due to Russo {\em et al}~\cite{Russo2}.
\begin{lemma}[Russo {\em et al}~\cite{Russo2}]\label{lemm3.1-1}
Let $\{X^\varepsilon,\;\varepsilon>0\}$ be a family of continuous processes. We suppose
\begin{itemize}
\item For any $\varepsilon>0$, the process $t\mapsto X^\varepsilon_t$ is increasing;
\item There is a continuous process $X=(X_t,t\geq 0)$ such that $X^\varepsilon_t\to X_t$ in probability as $\varepsilon$ goes to zero.
\end{itemize}
Then $Z^\varepsilon$ converges to $X$ ucp.
\end{lemma}
\begin{proposition}\label{prop4.1}
Let $a>-1$, $|b|<1$ and $|b|<1+a$. Then we have
\begin{equation}\label{sec4-eq4.111111}
[B^{a,b},B^{a,b}]^{(a,b)}_t=\kappa_{a,b}t^{1+a+b},
\end{equation}
where
$$
\kappa_{a,b}=\frac1{(1+b){\mathbb B}(a+1,b+1)}.
$$
\end{proposition}
Let now $(X,Y)$ be a $2$-dimensional normal random variable with the density
$$
\varphi(x,y)=\frac1{2\pi}e^{-\frac1{2\rho^2} (\sigma_2^2x^2-2\mu xy+\sigma_1^2y^2)},
$$
where $E[X]=E[Y]=0,\sigma_1^2=E[X^2],\sigma_2^2=E[Y^2],\mu= E[XY]$ and $\rho^2=\sigma_1^2\sigma_2^2-\mu^2$. Then, an elementary calculus can show that
\begin{equation}\label{sec4-eq4.1-11}
E[X^2Y^2]=E[X^2]E[Y^2]+2(E[XY])^{2}.
\end{equation}

Denote by
\begin{align*}
h_s(\varepsilon):=E(B^{a,b}_{s+\varepsilon}-B^{a,b}_s)^2 -\kappa_{a,b}\varepsilon^{1+b}s^a
\end{align*}
for $\varepsilon\in (0,1)$ and $s>0$. By making substitution $u-s=+v\varepsilon$ we have
\begin{align*}
h_s(\varepsilon)&=\frac{1}{{\mathbb B}(a+1,b+1)}\int_s^{s+\varepsilon}u^a(s+\varepsilon-u)^bdu -\kappa_{a,b}\varepsilon^{1+b}s^a\\
&=\frac{\varepsilon^{1+b}}{{\mathbb B}(a+1,b+1)}\int_0^{1}(s+\varepsilon v)^a(1-v)^bdu-\kappa_{a,b}\varepsilon^{1+b}s^a\\
&=\varepsilon^{1+b}\kappa_{a,b}\left((1+b)\int_0^{1}(s+\varepsilon v)^a(1-v)^bdu-s^a\right)\\
&=\frac{\varepsilon^{1+b}}{{\mathbb B}(a+1,b+1)}\int_0^{1}\left\{(s+\varepsilon v)^a-s^a\right\}(1-v)^bdu.
\end{align*}
Notice that
$$
\left|(s+\varepsilon v)^a-s^a\right|=(s+\varepsilon v)^a-s^a\leq
(\varepsilon v)^a
$$
for all $0<a\leq 1,\;\varepsilon>0,\;0\leq v\leq 1$, and
$$
\left|(s+\varepsilon v)^a-s^a\right|=(s+\varepsilon v)^a-s^a\leq
C_a\varepsilon v(s+\varepsilon v)^{a-1}\leq C_a\varepsilon v(s+\varepsilon)^{a-1}
$$
for all $a>1,\;\varepsilon>0,\;0\leq v\leq 1$, and
\begin{align*}
\left|(s+\varepsilon v)^a-s^a\right|&=s^a-(s+\varepsilon v)^a
=s^{a}\left(1-(\frac{s}{s+\varepsilon v})^{-a}\right)\\
&\leq s^{a}\left(1-\frac{s}{s+\varepsilon v}\right)=s^{a}\frac{\varepsilon v}{s+\varepsilon v}
\leq s^a\frac{\varepsilon v}{s^\nu(\varepsilon v)^{1-\nu}}
=(\varepsilon v)^{\nu}s^{a-\nu}
\end{align*}
for all $-1<a<0,\;\varepsilon>0,\;0\leq v\leq 1$ and all $0<\nu<1+a$ by Young's inequality. We get
\begin{equation}\label{sec6-eq6.6}
h_s(\varepsilon)\leq
\begin{cases}
C_{a,b}(s+\varepsilon)^{a-1}\varepsilon^{2+b},& {\text { if $a>1$}},\\
C_{a,b}\varepsilon^{1+b+a},& {\text { if $0<a\leq 1$}},\\
C_{a,b}s^{a-\nu}\varepsilon^{1+b+\nu},& {\text { if $-1<a<0$ \;($0<\nu<1+a$)}}
\end{cases}
\end{equation}
for all $s>0$ and $\varepsilon>0$.

\begin{proof}[Proof of Proposition~\ref{prop4.1}]
Denote
$$
X_\varepsilon(t)=\frac{1+a+b}{\varepsilon^{1+b}} \int_\varepsilon^{t+\varepsilon}
(B^{a,b}_{s+\varepsilon}-B^{a,b}_s)^2s^{b}ds
$$
for all $0<\varepsilon<1$. Then, it is sufficient to show that
$$
E\left(X_\varepsilon(t)-\kappa_{a,b}t^{1+a+b}\right)^2\longrightarrow 0,
$$
as $\varepsilon$ tends to zero. We have
\begin{align*}
X_\varepsilon(t)-\kappa_{a,b}t^{1+a+b} &=\frac{1+a+b}{\varepsilon^{1+b}}\left(\int_\varepsilon^{t+\varepsilon}
(B^{a,b}_{s+\varepsilon}-B^{a,b}_s)^2s^{b}ds -\kappa_{a,b}\frac{\varepsilon^{1+b}t^{1+a+b}}{1+a+b}\right)\\
&=\frac{1+a+b}{\varepsilon^{1+b}}\left(\int_\varepsilon^{t+\varepsilon}
(B^{a,b}_{s+\varepsilon}-B^{a,b}_s)^2s^{b}ds -\kappa_{a,b}\varepsilon^{1+b} \int_\varepsilon^{t+\varepsilon}s^{a+b}ds\right)\\
&\qquad+\kappa_{a,b}\left( (1+a+b) \int_\varepsilon^{t+\varepsilon}s^{a+b}ds- t^{1+a+b}\right)\\
&\equiv \Xi_1(a,b,\varepsilon,t)+\Xi_2(a,b,\varepsilon,t).
\end{align*}
We deduce from above
\begin{equation}\label{sec6-eq6.7}
E\left|X_\varepsilon(t)-\kappa_{a,b}t^{1+a+b}\right|^2
\leq 2E|\Xi_1(a,b,\varepsilon,t)|^2+2|\Xi_2(a,b,\varepsilon,t)|^2
\end{equation}
for $t\geq 0$ and $\varepsilon>0$. Clearly, we have
\begin{align}\notag
|\Xi_2(a,b,\varepsilon,t)|&=\kappa_{a,b}\left| (1+a+b)\int_\varepsilon^{t+\varepsilon}s^{a+b}ds-t^{1+a+b}\right|\\
\label{sec6-eq6.8}
&=\kappa_{a,b}\left|(t+\varepsilon)^{1+a+b}-\varepsilon^{1+a+b} -t^{1+a+b}\right|=O(\varepsilon^{1\wedge (1+a+b)})
\end{align}
for each $t\geq 0$.

Now, let us estimate $E|\Xi_1(a,b,\varepsilon,t)|^2$. We have
\begin{equation}\label{sec6-eq6.9}
\begin{split}
E|\Xi_1(a,b,\varepsilon,t)|^2&=\frac{(1+a+b)^2}{\varepsilon^{2+2b}} E\left(\int_\varepsilon^{t+\varepsilon}
\left\{(B^{a,b}_{s+\varepsilon}-B^{a,b}_s)^2-\kappa_{a,b} \varepsilon^{1+b}s^{a}\right\}s^{b}ds\right)^2\\
&=\frac{(1+a+b)^2}{\varepsilon^{2+2b}} \int_\varepsilon^{t+\varepsilon} \int_\varepsilon^{t+\varepsilon}A_\varepsilon(s,r)(sr)^{b}dsdr
\end{split}
\end{equation}
for each $t\geq 0$, where
\begin{align*}
A_\varepsilon(s,r):&=E\left(
(B^{a,b}_{s+\varepsilon}-B^{a,b}_s)^2-\kappa_{a,b} \varepsilon^{1+b}s^a\right)\left(
(B^{a,b}_{r+\varepsilon}-B^{a,b}_r)^2-\kappa_{a,b} \varepsilon^{1+b}r^a\right)\\
&=E(B^{a,b}_{s+\varepsilon}-B^{a,b}_s)^2(B^{a,b}_{r+\varepsilon} -B^{a,b}_r)^2+(\kappa_{a,b})^2\varepsilon^{2+2b}(sr)^a\\
&\qquad-\kappa_{a,b}\varepsilon^{1+b}
E\left((B^{a,b}_{r+\varepsilon}-B^{a,b}_r)^2s^a +(B^{a,b}_{s+\varepsilon}-B^{a,b}_s)^2r^a\right).
\end{align*}
For all $r,s\geq 0$ and $\varepsilon>0$ by decomposing
\begin{align*}
E(B^{a,b}_{s+\varepsilon}-B^{a,b}_s)^2 =h_s(\varepsilon)+\kappa_{a,b}\varepsilon^{1+b}s^a
\end{align*}
we have
\begin{align*}
E[(B^{a,b}_{s+\varepsilon}-B^{a,b}_s)^2(B^{a,b}_{r+\varepsilon} -&B^{a,b}_r)^2]=
E\left[(B^{a,b}_{s+\varepsilon}-B^{a,b}_s)^2\right]
E\left[(B^{a,b}_{r+\varepsilon}-B^{a,b}_r)^2\right]\\
&\hspace{2cm}+2\left(E\left[(B^{a,b}_{s+\varepsilon}-B^{a,b}_s)
(B^{a,b}_{r+\varepsilon}-B^{a,b}_r)\right]\right)^{2}\\
&=\left[h_s(\varepsilon)+\kappa_{a,b}\varepsilon^{1+b}s^a\right]
\left[h_r(\varepsilon)+\kappa_{a,b}\varepsilon^{1+b}s^a \right]+2(\mu_{s,r})^2
\end{align*}
by~\eqref{sec4-eq4.1-11}, where $\mu_{s,r}:=E\left[(B^{a,b}_{s+\varepsilon}-B^{a,b}_s)
(B^{a,b}_{r+\varepsilon}-B^{a,b}_r)\right]$, which yields
\begin{equation}\label{sec6-eq6.10}
A_\varepsilon(s,r)=h_s(\varepsilon)h_r(\varepsilon)+2(\mu_{s,r})^2.
\end{equation}
On the other hand, for $0<s-r<\varepsilon\leq 1$ we have
\begin{align*}
(\mu_{s,r})^2&\leq E[(B^{a,b}_{s+\varepsilon}-B^{a,b}_s)^2]E[(B^{a,b}_{r+\varepsilon} -B^{a,b}_{r})^2]\\
&\leq (s+\varepsilon)^a(r+\varepsilon)^a\varepsilon^{2+2b}
\leq (s+\varepsilon)^a(r+\varepsilon)^a\frac{\varepsilon^{2+2b+\gamma} }{(s-r)^\gamma}
\end{align*}
for all $0\leq \gamma\leq 1$. It follows from Lemma~\ref{lem3.3} that
\begin{equation}\label{sec6-eq6.11}
\begin{split}
(\mu_{s,r})^2&\leq C_{a,b}\left(r^a\vee s^a\right)^{2\alpha} [(s+\varepsilon)(r+\varepsilon)]^{a(1-\alpha)} \frac{\varepsilon^{2+2b+2\alpha(1-b)} }{(s-r)^{2(1+b)\alpha}}1_{\{s-r>\varepsilon\}}\\
&\qquad+(s+\varepsilon)^a(r+\varepsilon)^a \frac{\varepsilon^{2+2b+\alpha} }{(s-r)^\alpha}1_{\{s-r\leq \varepsilon\}}
\end{split}
\end{equation}
for all $0<\alpha<\frac1{2(1+b)}\wedge 1$. Combining this with~\eqref{sec6-eq6.6},~\eqref{sec6-eq6.7},~\eqref{sec6-eq6.8}, ~\eqref{sec6-eq6.9} and~\eqref{sec6-eq6.10} we see that
\begin{align*}
E|\Xi_1(a,b,\varepsilon,t)|^2&=\frac{(1+a+b)^2}{\varepsilon^{2+2b}} \int_\varepsilon^{t+\varepsilon} \int_\varepsilon^{t+\varepsilon}A_\varepsilon(s,r)(sr)^{b}drds \longrightarrow 0
\end{align*}
for all $t\geq 0$, as $\varepsilon$ tends to zero, and the proposition follows.
\end{proof}

Recall that the local H\"{o}lder index $\gamma_0$
of a continuous paths process $\{X_t: t\geq 0\}$ is the supremum of
the exponents $\gamma$ verifying, for any $T>0$:
$$
P(\{\omega: \exists L(\omega)>0, \forall s,t \in[0,T],
|X_t(\omega)-X_s(\omega)|\leq L(\omega)|t-s|^\gamma\})=1.
$$
The next lemma is considered by Gradinaru-Nourdin~\cite{Grad3}.

\begin{lemma}\label{Grad-Nourdin}
Let $g:{\mathbb R}\to {\mathbb R}$ be a continuous function satisfying
\begin{equation}\label{eq4.2-Gradinaru--Nourdin}
|g(x)-g(y)|\leq C|x-y|^a(1+x^2+y^2)^b,\quad (C>0,0<a\leq 1,b>0),
\end{equation}
for all $x,y\in {\mathbb R}$ and let $X$ be a locally H\"older
continuous paths process with index $\gamma\in (0,1)$. Assume that
$V$ is a bounded variation continuous paths process. Set
$$
X^{g}_\varepsilon(t)=\int_\varepsilon^{t+\varepsilon} g(\frac{X_{s+\varepsilon}-X_s
}{\varepsilon^\gamma})ds
$$
for $t\geq 0$, $\varepsilon>0$. If for each $t\geq 0$, as
$\varepsilon\to 0$,
\begin{equation}\label{condition}
\|X^{g}_\varepsilon(t)-V_t\|_{L^2}^2=O(\varepsilon^\alpha)
\end{equation}
with $\alpha>0$, then, for any $t\geq 0$, $\lim_{\varepsilon\to
0}X^{g}_\varepsilon(t)=V_t$ almost surely, and if $g$ is
non-negative, for any continuous stochastic process $\{Y_t:\;t\geq
0\}$,
\begin{equation}
\lim_{\varepsilon\to 0}
\int_\varepsilon^{t+\varepsilon} Y_sg(\frac{X_{s+\varepsilon}-X_s}{\varepsilon^\gamma})ds
\longrightarrow \int_0^tY_sdV_s,
\end{equation}
almost surely, uniformly in $t$ on each compact interval.
\end{lemma}
Recall that
$$
X_\varepsilon(t)=\frac{1+a+b}{\varepsilon^{1+b}} \int_\varepsilon^{t+\varepsilon}
(B^{a,b}_{s+\varepsilon}-B^{a,b}_s)^2s^{b}ds
$$
for all $0<\varepsilon<1$. The proof of Proposition~\ref{prop4.1} points out that there exists $\beta\in (0,1)$ such that
\begin{align*}
\|X_\varepsilon(t)-\kappa_{a,b}t^{1+a+b}\|_{L^2}=O(\varepsilon^{\beta}), \qquad \varepsilon\to 0
\end{align*}
for all $t\geq 0$. Notice that $g(x)=x^2$ satisfies the condition~\eqref{eq4.2-Gradinaru--Nourdin}. We obtain the convergence
\begin{align*}
\lim_{\varepsilon\downarrow 0}\frac{1+a+b}{\varepsilon^{1+b}}
\int_\varepsilon^{t+\varepsilon} f(B^{a,b}_s)(B^{a,b}_{s+\varepsilon}-B^{a,b}_s)^2s^{b}ds
\longrightarrow \kappa_{a,b}\int_0^tf(B^{a,b}_s)ds^{1+a+b},
\end{align*}
almost surely, uniformly in $t$ on each compact interval by taking $Y_s=f(B^{a,b}_s)s^{b}$ for $s\geq 0$. Clearly, by the H\"older continuity of weighted-fBm $B^{a,b}$, we get
\begin{align*}
\frac{1+a+b}{\varepsilon^{1+b}}
\int_\varepsilon^{t+\varepsilon}o(B^{a,b}_{s+\varepsilon}-B^{a,b}_s)
(B^{a,b}_{s+\varepsilon}-B^{a,b}_s)s^{b}ds\longrightarrow 0
\end{align*}
in $L^1(\Omega)$, as $\varepsilon \to 0$. Thus, we have obtained the next result.

\begin{corollary}
Let $a>-1$, $|b|<1$, $|b|<1+a$ and let $f\in
C^1({\mathbb R})$. Then the generalized quadratic covariation $[f(B^{a,b}),B^{a,b}]^{(a,b)}$ exists and
\begin{align}\label{sec6-eq6.2}
[f(B^{a,b}),B^{a,b}]^{(a,b)}_t&=\kappa_{a,b} \int_0^tf'(B^{a,b}_s)ds^{1+a+b}.
\end{align}
\end{corollary}

%%%%%%%%%%%%%%%%%%%%%%%%%%%%%%%%%%%%%%%%%%%%%%%%%%%%%%%%%%%%%%%%%%
%%%%%%%%%%%%%%%%%%%%%%%%%%%%%%%%%%%%%%%%%%%%%%%%%%%%%%%%%%%%%%%%%%
%%%%%%%%%%%%%%%%%%%%%%%%%%%%%%%%%%%%%%%%%%%%%%%%%%%%%%%%%%%%%%%%%%
\section{The generalized quadratic covariation, existence}~\label{sec5}

In this section, for $-1<b<0$ we shall study the existence of the
{\it generalized quadratic covariation} more general class of
functions than the one of class $C^1$. Recall that
$$
J_\varepsilon(f,t):=\frac{1+a+b}{\varepsilon^{1+b}} \int_\varepsilon^{t+\varepsilon}\left\{ f(B^{a,b}_{s+\varepsilon})
-f(B^{a,b}_s)\right\}(B^{a,b}_{s+\varepsilon}-B^{a,b}_s)s^{b}ds
$$
for all $\varepsilon>0$ and all Borel functions $f$. Consider the decomposition
\begin{equation}\label{sec4-eq4.000000}
\begin{split}
\frac{1}{\varepsilon^{1+b}}\int_\varepsilon^{t+\varepsilon} &\left\{f(B^{a,b}_{
s+\varepsilon})-f(B^{a,b}_s)\right\}(B^{a,b}_{s+\varepsilon} -B^{a,b}_s)s^{b}ds\\
&=\frac{1}{\varepsilon^{1+b}}\int_\varepsilon^{t+\varepsilon} f(B^{a,b}_{
s+\varepsilon})(B^{a,b}_{s+\varepsilon}-B^{a,b}_s)s^bds\\
&\qquad\qquad\qquad-\frac{1}{\varepsilon^{1+b}}
\int_\varepsilon^{t+\varepsilon} f(B^{a,b}_s)(B^{a,b}_{s+\varepsilon}-B^{a,b}_s)s^bds\\
&\equiv I_\varepsilon^{+}(f,t)-I_\varepsilon^{-}(f,t),
\end{split}
\end{equation}
and define the set ${\mathscr H}$ of measurable functions $f$ on
${\mathbb R}$ such that $\|f\|_{{\mathscr H}}<\infty$, where
\begin{align*}
\|f\|_{{\mathscr H}}:=&\sqrt{\int_0^{T+1}\int_{\mathbb
R}|f(x)|^2e^{-\frac{x^2}{2s^{1+a+b}
}}\frac{dxds}{\sqrt{2\pi}s^{(1-a-b)/2}}}.
\end{align*}
For the set ${\mathscr H}$ we have
\begin{itemize}
\item [(1)] ${\mathscr H}$ is a Banach space.
\item [(2)] ${\mathscr H}\supset C^\infty_0({\mathbb R})$, the set of infinitely differentiable functions with compact support.
\item [(3)] the set ${\mathscr E}$ of elementary functions of the form
$$
f_\triangle(x)=\sum_{i}f_{i}1_{(x_{i-1},x_{i}]}(x)
$$
is dense in ${\mathscr H}$, where $\{x_i,0\leq i\leq l\}$ is an
finite sequence of real numbers such that $x_i<x_{i+1}$.
\item [(4)] the space ${\mathscr H}$ contains all Borel functions $f$ satisfying
    $$
    |f(x)|\leq Ce^{\beta x^2},\qquad x\in {\mathbb R}
    $$
    with $\beta<\frac14T^{-(1+a+b)}$.
\end{itemize}

For simplicity we let $T=1$ in the following discussions. The main result of this section is to explain and prove the
following theorem.
\begin{theorem}\label{th6.2}
Let $a>-1$, $-1<b<0$, $|b|<1+a$ and $f\in {\mathscr H}$. Then the generalized quadratic covariation $[f(B^{a,b}),B^{a,b}]^{(a,b)}$ exists, and we have
\begin{align}\label{th4.1-eq}
E\left|[f(B^{a,b}),B^{a,b}]^{(a,b)}_t\right|^2\leq C_{a,b,T}\|f\|_{{\mathscr H}}^2
\end{align}
for all $a\geq 0,t\in [0,1]$.
\end{theorem}

For simplicity we let $T=1$ in the rest of this section.
\begin{lemma}\label{lem6.2}
Let $a>-1,\;-1<b<0,\;|b|<1+a$ and let $f\in {\mathscr H}$. We then have
\begin{align}
&E\left|I_\varepsilon^{-}(f,t)\right|^2\leq C_{a,b}\|f\|_{{\mathscr H}}^2,\\
&E\left|I_\varepsilon^{+}(f,t)\right|^2\leq C_{a,b}\|f\|_{{\mathscr H}}^2
\end{align}
for all $0<\varepsilon<1$ and $t\in [0,1]$.
\end{lemma}
\begin{proof}
We prove only the first estimate, and similarly one can prove the second estimate. By the denseness of ${\mathscr E}$ in ${\mathscr H}$ we only need to show that the lemma holds for all $f\in {\mathscr E}$. Consider the function $\zeta$ on ${\mathbb R}$ by
\begin{equation}
\zeta(x):=
\begin{cases}
ce^{\frac1{(x-1)^2-1}}, &{\text { $x\in (0,2)$}},\\
0, &{\text { otherwise}},
\end{cases}
\end{equation}
where $c$ is a normalizing constant such that $\int_{\mathbb
R}\zeta(x)dx=1$, and define the mollifiers
\begin{equation}\label{sec7-eq7.4}
\zeta_n(x):=n\zeta(nx),\qquad n=1,2,\ldots.
\end{equation}
For $f_\triangle\in {\mathscr E}$ consider the sequence of smooth functions
$$
f_{\triangle,n}(x):=\int_{\mathbb R}f_{\triangle}(x-{y})\zeta_n(y)dy=\int_0^2 f_{\triangle}(x-\frac{y}n)\zeta(y)dy,\quad x\in {\mathbb R}.
$$
Then $f_{\triangle,n}\in C_0^\infty({\mathbb R})$ for all $n\geq 1$ and $f_{\triangle,n}\to f_{\triangle}$ in ${\mathscr H}$, as $n$ tends to infinity. Thus, by approximating we may assume that $f$ is an infinitely differentiable function with compact support. It follows from the duality relationship that
\begin{equation}\label{sec4-eq4.800}
\begin{split}
E[f(&B^{a,b}_{s})f(B^{a,b}_{r})(B^{a,b}_{s+\varepsilon} -B^{a,b}_s)(B^{a,b}_{r+\varepsilon}-B^{a,b}_r)]\\
&=E\left[f(B^{a,b}_{s})f(B^{a,b}_{
r})(B^{a,b}_{s+\varepsilon}-B^{a,b}_s) \int_r^{r+\varepsilon}dB^{a,b}_l\right]\\
&=E\left\langle D^{a,b}\left(
f(B^{a,b}_{s})f(B^{a,b}_{r})(B^{a,b}_{s+\varepsilon}-B^{a,b}_s)\right), 1_{[r,r+\varepsilon]}\right\rangle_{{\mathcal H}_{a,b}}\\
&=E\left[B^{a,b}_{s}(B^{a,b}_{r+\varepsilon}-B^{a,b}_r)\right]E\left[
f'(B^{a,b}_{s})f(B^{a,b}_{r})(B^{a,b}_{s+\varepsilon}-B^{a,b}_s) \right]\\
&\qquad +E\left[B^{a,b}_{r}(B^{a,b}_{r+\varepsilon}-B^{a,b}_r)\right]E\left[
f(B^{a,b}_{s})f'(B^{a,b}_{r})(B^{a,b}_{s+\varepsilon}-B^{a,b}_s) \right]\\
&\qquad
+E\left[(B^{a,b}_{r+\varepsilon}-B^{a,b}_r)(B^{a,b}_{s+\varepsilon} -B^{a,b}_s) \right]E\left[
f(B^{a,b}_{s})f(B^{a,b}_{r})\right]\\
&=E\left[B^{a,b}_{s}(B^{a,b}_{r+\varepsilon}-B^{a,b}_r)\right]
E\left[B^{a,b}_{r}(B^{a,b}_{s+\varepsilon}-B^{a,b}_s)\right]
E\left[f'(B^{a,b}_{s})f'(B^{a,b}_{r})\right]\\
&\qquad+E\left[B^{a,b}_{s}(B^{a,b}_{r+\varepsilon}-B^{a,b}_r)\right]
E\left[B^{a,b}_{s}(B^{a,b}_{s+\varepsilon}-B^{a,b}_s)\right]
E\left[f''(B^{a,b}_{s})f(B^{a,b}_{r})\right]\\
&\qquad+E\left[B^{a,b}_{r}(B^{a,b}_{r+\varepsilon}-B^{a,b}_r)\right]
E\left[B^{a,b}_{s}(B^{a,b}_{s+\varepsilon}-B^{a,b}_s)\right]
E\left[f'(B^{a,b}_{s})f'(B^{a,b}_{r})\right]\\
&\qquad+E\left[B^{a,b}_{r}(B^{a,b}_{r+\varepsilon}-B^{a,b}_r)\right]
E\left[B^{a,b}_{r}(B^{a,b}_{s+\varepsilon}-B^{a,b}_s)\right]
E\left[f(B^{a,b}_{s})f''(B^{a,b}_{r})\right]\\
&\qquad
+E\left[(B^{a,b}_{r+\varepsilon}-B^{a,b}_r)(B^{a,b}_{s+\varepsilon} -B^{a,b}_s)\right]E\left[f(B^{a,b}_{s})f(B^{a,b}_{r})\right]\\
&\equiv \sum_{i=1}^5 \Psi_i(s,r,\varepsilon).
\end{split}
\end{equation}
In order to end the proof we claim now that
\begin{equation}\label{eq4.4}
\frac{1}{\varepsilon^{2+2b}}\left| \int_\varepsilon^{t+\varepsilon}\int_\varepsilon^{t+\varepsilon} \Psi_i(s,r,\varepsilon)(sr)^{b}dsdr \right|\leq C_{a,b}\|f\|^2_{{\mathscr H}},\qquad i=1,2,\ldots,5
\end{equation}
for all $\varepsilon>0$ small enough.

{\bf For $i=5$} we have
\begin{align*}
|E[(B^{a,b}_{r+\varepsilon}-B^{a,b}_r)(B^{a,b}_{s+\varepsilon} -B^{a,b}_s)]|&\leq \sqrt{E[(B^{a,b}_{r+\varepsilon}-B^{a,b}_r)^2E(B^{a,b}_{s+\varepsilon} -B^{a,b}_s)^2]}\\
&\leq C_{a,b}[(r+\varepsilon)(s+\varepsilon)]^{a/2} \varepsilon^{1+b}\\
&\leq C_{a,b}[(r+\varepsilon)(s+\varepsilon)]^{a/2}\frac{ \varepsilon^{2+2b}}{(s-r)^{1+b}}
\end{align*}
for $0<s-r<\varepsilon\leq 1$, which gives
\begin{align*}
I_\varepsilon:&= \frac{1}{\varepsilon^{2+2b}} \int_\varepsilon^{t+\varepsilon}\int_{s-\varepsilon}^s
|E[(B^{a,b}_{r+\varepsilon}-B^{a,b}_r)(B^{a,b}_{s+\varepsilon} -B^{a,b}_s)]|E[f^2(B^{a,b}_{s})](sr)^{b}
drds\\
&\leq C_{a,b}\int_\varepsilon^{t+\varepsilon}\int_{s-\varepsilon}^s \frac{(r^a\vee s^a)(sr)^{b}}{(s-r)^{1+b}} E[f^2(B^{a,b}_{s})]drds\\
&\leq C_{a,b}\int_\varepsilon^{t+\varepsilon}
s^bEf^2(B^{a,b}_{s})ds
\int_0^s \frac{r^{b}(r^a\vee s^a)}{(s-r)^{1+b}} dr\\
&=C_{a,b}\int_\varepsilon^{t+\varepsilon}s^{b+a}Ef^2(B^{a,b}_{s})ds \leq C_{a,b}\int_0^{2}s^{a+b}Ef^2(B^{a,b}_{s})ds
\end{align*}
for $0<s-r<\varepsilon\leq 1$. Moreover, by~\eqref{sec3-eq3.8==} with $\beta=1$ and $\nu=-b$ we have
\begin{align*}
II_\varepsilon:&=\frac{1}{\varepsilon^{2+2b}} \int_{\varepsilon}^{t+\varepsilon} \int_\varepsilon^{s-\varepsilon}
|E[(B^{a,b}_{r+\varepsilon}-B^{a,b}_r)(B^{a,b}_{s+\varepsilon} -B^{a,b}_s)]|E[f^2(B^{a,b}_{s})] (sr)^{b}drds  \\
&\leq C_{a,b}\int_{\varepsilon}^{t+\varepsilon}
s^bE[f^2(B^{a,b}_{s})]ds\int_\varepsilon^{s-\varepsilon}(r^a\vee s^a)\frac{r^bdr}{(s-r)^{1+b}} \\
&\leq C_{a,b} \int_{\varepsilon}^{t+\varepsilon}s^b Ef^2(B^{a,b}_{s})ds \int_0^s
\frac{r^{b}(r^a\vee s^a)}{(s-r)^{1+b}}dr\\
&\leq C_{a,b}\int_{\varepsilon}^{t+\varepsilon}s^{a+b} Ef^2(B^{a,b}_{s})ds\leq C_{a,b}\int_{0}^{2}s^{a+b} Ef^2(B^{a,b}_{s})ds
\end{align*}
for all $0<\varepsilon\leq 1$. It follows that
\begin{align*}
& \frac{1}{\varepsilon^{2+2b}} \int_\varepsilon^{t+\varepsilon}\int_\varepsilon^{t+\varepsilon} |\Psi_5(s,r,\varepsilon)|(sr)^{b} drds \\
&\leq \frac{1 }{2\varepsilon^{2+2b}} \int_\varepsilon^{t+\varepsilon}\int_\varepsilon^{t+\varepsilon}
|E[(B^{a,b}_{r+\varepsilon}-B^{a,b}_r)(B^{a,b}_{s+\varepsilon} -B^{a,b}_s)]|[Ef^2(B^{a,b}_{s})+Ef^2(B^{a,b}_{r})] (sr)^{b}drds\\
&\leq \frac{1}{\varepsilon^{2+2b}}\int_\varepsilon^{t+\varepsilon} \int_\varepsilon^{t+\varepsilon}
|E[(B^{a,b}_{r+\varepsilon}-B^{a,b}_r)(B^{a,b}_{s+\varepsilon} -B^{a,b}_s)]|Ef^2(B^{a,b}_{s})(sr)^{b}drds\\
&\leq \frac{2}{\varepsilon^{2+2b}}\int_\varepsilon^{t+\varepsilon} \int_\varepsilon^{s}
|E[(B^{a,b}_{r+\varepsilon}-B^{a,b}_r)(B^{a,b}_{s+\varepsilon} -B^{a,b}_s)]|Ef^2(B^{a,b}_{s})(sr)^{b}drds\\
&\leq 2(I_\varepsilon+II_\varepsilon)\leq C_{a,b}
\int_{0}^{2}s^{a+b} E[f^2(B^{a,b}_{s})]ds
\end{align*}
for all $0<\varepsilon\leq 1$.

{\bf For $i=1$}, by Lemma~\ref{lem3.4} and Lemma~\ref{lem3.5} we have
\begin{align*}
\frac{1}{\varepsilon^{2+2b}}& \int_\varepsilon^{t+\varepsilon}\int_\varepsilon^{t+\varepsilon} |\Psi_1(s,r,\varepsilon)|(sr)^{b}drds\\
&\leq  \frac{C_{a,b}}{2\varepsilon^{2+2b}} \int_\varepsilon^{t+\varepsilon}\int_\varepsilon^{t+\varepsilon}
|E[B^{a,b}_{s}(B^{a,b}_{r+\varepsilon}-B^{a,b}_r)]
E[B^{a,b}_{r}(B^{a,b}_{s+\varepsilon}-B^{a,b}_s)]|\\
&\qquad\qquad\cdot
\left(\frac{(rs)^{\frac12(1+a+b)}}{\rho^2_{s,r}} +\frac{\mu_{s,r}}{\rho^2_{s,r}}\right)\left(Ef^2(B^{a,b}_s)
+Ef^2(B^{a,b}_r)\right)(sr)^{b}dsdr\\
&= \frac{C_{a,b}}{\varepsilon^{2+2b}} \int_\varepsilon^{t+\varepsilon}\int_\varepsilon^{t+\varepsilon}
|E[B^{a,b}_{s}(B^{a,b}_{r+\varepsilon}-B^{a,b}_r)]
E[B^{a,b}_{r}(B^{a,b}_{s+\varepsilon}-B^{a,b}_s)]|\\
&\qquad\qquad\cdot
\left(\frac{(rs)^{\frac12(1+a+b)}}{\rho^2_{s,r}} +\frac{\mu_{s,r}}{\rho^2_{s,r}}\right) Ef^2(B^{a,b}_s)(sr)^{b}dsdr\\
&\leq  C_{a,b}\int_\varepsilon^{t+\varepsilon}s^{a+b}ds\int_\varepsilon^{s}
\left(\frac{(rs)^{\frac12(1+a+b)}}{\rho^2_{s,r}} +\frac{\mu_{s,r}}{\rho^2_{s,r}}\right)Ef^2(B^{a,b}_s)r^{a+b}dr\\
&\qquad +C_{a,b}\int_\varepsilon^{t+\varepsilon}s^{a+b}ds \int_s^{s+\varepsilon}
\left(\frac{(rs)^{\frac12(1+a+b)}}{\rho^2_{s,r}} +\frac{\mu_{s,r}}{\rho^2_{s,r}}\right)Ef^2(B^{a,b}_s)r^{a+b}dr\\
&\qquad +C_{a,b}\int_\varepsilon^{t+\varepsilon}s^{2a+b}ds \int_{s+\varepsilon}^{t+\varepsilon}
\left(\frac{(rs)^{\frac12(1+a+b)}}{\rho^2_{s,r}} +\frac{\mu_{s,r}}{\rho^2_{s,r}}\right)Ef^2(B^{a,b}_s)r^{b}dr\\
&\leq C_{a,b}
\int_{0}^{2}s^{a+b} E[f^2(B^{a,b}_{s})]ds
\end{align*}
for all $0<\varepsilon\leq 1$. Similarly, we can obtain the estimate~\eqref{eq4.4} for $i=2,3,4$, and
the lemma follows.
\end{proof}

\begin{lemma}\label{lem6.3}
Let $a>-1,\;-1<b<0,\;|b|<1+a$ and let $f$ be an infinitely differentiable function with compact support. We then have
\begin{align}
\Pi_1(t,\varepsilon_i,\varepsilon_j):= \frac1{\varepsilon_i^{2+2b}} E\Bigl|\int_{t+\varepsilon_2}^{t+\varepsilon_1}f(B^{a,b}_s) (B^{a,b}_{s+\varepsilon_i}-B^{a,b}_s)s^bds\Bigr|^2\longrightarrow 0
\end{align}
for all $0<\varepsilon_2<\varepsilon_1<1$ and $t\in [0,1]$, as $\varepsilon_1\to 0$, where $i,j=1,2$ and $i\neq j$.
\end{lemma}
\begin{proof}
From the proof of Lemma~\ref{lem6.2} one can easily prove the result.
\end{proof}

Now we can show our main result.
\begin{proof}[Proof of Theorem~\ref{th6.2}]
From Lemma~\ref{lem6.2}, it is enough to show that $I_{\varepsilon_1}^{-}(f,t)$ and $I_{\varepsilon_1}^{+}(f,t)$ are two Cauchy's sequences in $L^2(\Omega)$ for all $t\in [0,1]$. That is,
\begin{equation}\label{sec40-eq3-1}
E\left|I_{\varepsilon_1}^{-}(f,t)-I_{\varepsilon_2}^{-}(f,t)\right|^2
\longrightarrow 0,
\end{equation}
and
\begin{equation}\label{sec40-eq3-2}
E\left|I_{\varepsilon_1}^{+}(f,t)-I_{\varepsilon_2}^{+}(f,t)\right|^2
\longrightarrow 0
\end{equation}
for all $t\in [0,1]$, as $\varepsilon_1,\varepsilon_2\downarrow 0$. We prove only the convergence~\eqref{sec40-eq3-1}, and similarly one can prove~\eqref{sec40-eq3-2}. Without loss of generality one may assume that $\varepsilon_1>\varepsilon_2$. It follows that
\begin{align*}
&E\bigl|I_{\varepsilon_1}^{-}(f,t)-I_{\varepsilon_2}^{-}(f,t)\bigr|^2
\\
&\leq 3E\Bigl|\frac1{\varepsilon_1^{1+b}} \int_{t+\varepsilon_2}^{t+\varepsilon_1}f(B^{a,b}_s) (B^{a,b}_{s+\varepsilon_1}-B^{a,b}_s)s^bds\Bigr|^2\\
&\qquad+3E\Bigl|\frac1{\varepsilon_2^{1+b}} \int_{\varepsilon_2}^{\varepsilon_1}f(B^{a,b}_s) (B^{a,b}_{s+\varepsilon_2}-B^{a,b}_s)s^bds\Bigr|^2\\
&+3E\Bigl|\frac1{\varepsilon_1^{1+b}} \int_{\varepsilon_2}^{t+\varepsilon_1}f(B^{a,b}_s) (B^{a,b}_{s+\varepsilon_1}-B^{a,b}_s)s^bds- \frac1{\varepsilon_2^{1+b}} \int_{\varepsilon_2}^{t+\varepsilon_1}f(B^{a,b}_s) (B^{a,b}_{s+\varepsilon_2}-B^{a,b}_s)s^bds\Bigr|^2\\
&\equiv \Pi_1(t,\varepsilon_1,\varepsilon_2) +\Pi_1(0,\varepsilon_2,\varepsilon_1) +\Pi_2(t,\varepsilon_1,\varepsilon_2).
\end{align*}
In order to end the proof, it is enough to check $\Pi_2(t,\varepsilon_1,\varepsilon_2)\to 0$, as $\varepsilon_1,\varepsilon_2\to 0$. We have
\begin{align*}
\Pi_2&(t,\varepsilon_1,\varepsilon_2)\\
&=\frac1{\varepsilon_1^{2+2b}} \int_{\varepsilon_2}^{t+\varepsilon_1} \int_{\varepsilon_2}^{t+\varepsilon_1} Ef(B^{a,b}_s)f(B^{a,b}_r)(B^{a,b}_{s+\varepsilon_1} -B^{a,b}_s) (B^{a,b}_{r+\varepsilon_1}-B^{a,b}_r)(sr)^bdrds\\
&\quad-
\frac2{\varepsilon_1^{1+b}\varepsilon_2^{1+b}} \int_{\varepsilon_2}^{t+\varepsilon_1} \int_{\varepsilon_2}^{t+\varepsilon_1} Ef(B^{a,b}_s)f(B^{a,b}_r)(B^{a,b}_{s+\varepsilon_1} -B^{a,b}_s) (B^{a,b}_{r+\varepsilon_2}-B^{a,b}_r)(sr)^bdrds\\
&\quad+\frac1{\varepsilon_2^{2+2b}} \int_{\varepsilon_2}^{t+\varepsilon_1} \int_{\varepsilon_2}^{t+\varepsilon_1} Ef(B^{a,b}_s)f(B^{a,b}_r)(B^{a,b}_{s+\varepsilon_2}-B^{a,b}_s) (B^{a,b}_{r+\varepsilon_2}-B^{a,b}_r)(sr)^bdrds\\
&\equiv \frac1{\varepsilon_1^{2+2b}\varepsilon_2^{1+b}} \int_{\varepsilon_2}^{t+\varepsilon_1} \int_{\varepsilon_2}^{t+\varepsilon_1}\left\{\varepsilon_2^{1+b} \Phi_{s,r}(1,\varepsilon_1)-\varepsilon_1^{1+b}
\Phi_{s,r}(2,\varepsilon_1,\varepsilon_2)\right\}(sr)^bdrds\\
&\quad+
\frac1{\varepsilon_1^{1+b}\varepsilon_2^{2+2b}} \int_{\varepsilon_2}^{t+\varepsilon_1} \int_{\varepsilon_2}^{t+\varepsilon_1} \left\{\varepsilon_1^{1+b}\Phi_{s,r}(1,\varepsilon_2)
-\varepsilon_2^{1+b}
\Phi_{s,r}(2,\varepsilon_1,\varepsilon_2)\right\}(sr)^bdrds,
\end{align*}

where
$$
\Phi_{s,r}(1,\varepsilon) =E\left[f(B^{a,b}_s)f(B^{a,b}_r)(B^{a,b}_{s+\varepsilon}-B^{a,b}_s) (B^{a,b}_{r+\varepsilon}-B^{a,b}_r)\right],
$$
and
$$
\Phi_{s,r}(2,\varepsilon_1,\varepsilon_2) =E\left[f(B^{a,b}_s)f(B^{a,b}_r)(B^{a,b}_{s+\varepsilon_1}-B^{a,b}_s) (B^{a,b}_{r+\varepsilon_2}-B^{a,b}_r)\right].
$$
In order to estimate $\Phi_{s,r}(1,\varepsilon)$ and $\Phi_{s,r}(2,\varepsilon_1,\varepsilon_2)$, by approximating we may assume that $f$ is an infinitely differentiable function with compact support. We then have by~\eqref{sec4-eq4.800},
\begin{align*}
\Phi_{s,r}(1,\varepsilon)&=\sum_{i=1}^5\Psi_i(s,r,{\varepsilon})\\
&=E\left[B^{a,b}_{s}(B^{a,b}_{r+\varepsilon}-B^{a,b}_r)\right]
E\left[B^{a,b}_{s}(B^{a,b}_{s+\varepsilon}-B^{a,b}_s)\right]
E\left[f''(B^{a,b}_{s})f(B^{a,b}_{r})\right]\\
&+E\left[B^{a,b}_{s}(B^{a,b}_{r+\varepsilon}-B^{a,b}_r)\right]
E\left[B^{a,b}_{r}(B^{a,b}_{s+\varepsilon}-B^{a,b}_s)\right]
E\left[f'(B^{a,b}_{s})f'(B^{a,b}_{r})\right]\\
&\qquad+E\left[B^{a,b}_{r}(B^{a,b}_{r+\varepsilon}-B^{a,b}_r)\right] E\left[B^{a,b}_{s}(B^{a,b}_{s+\varepsilon}-B^{a,b}_s)\right]
E\left[f'(B^{a,b}_{s})f'(B^{a,b}_{r})\right]\\
&\qquad+E\left[B^{a,b}_{r}(B^{a,b}_{r+\varepsilon}-B^{a,b}_r)\right]
E\left[B^{a,b}_{r}(B^{a,b}_{s+\varepsilon}-B^{a,b}_s)\right]
E\left[f(B^{a,b}_{s})f''(B^{a,b}_{r})\right]\\
&\qquad\qquad+E\left[(B^{a,b}_{r+\varepsilon}-B^{a,b}_r)
(B^{a,b}_{s+\varepsilon}-B^{a,b}_s)\right]E\left[
f(B^{a,b}_{s})f(B^{a,b}_{r})\right]
\end{align*}
and
\begin{align*}
\Phi_{s,r}(2,\varepsilon_1,\varepsilon_2)&=
E\left[B^{a,b}_{s}(B^{a,b}_{r+\varepsilon_2}-B^{a,b}_r)\right]E\left[
f'(B^{a,b}_{s})f(B^{a,b}_{r})(B^{a,b}_{s+\varepsilon_1}-B^{a,b}_s) \right]\\
&\qquad +E\left[B^{a,b}_{r}(B^{a,b}_{r+\varepsilon_2}-B^{a,b}_r)\right]E\left[
f(B^{a,b}_{s})f'(B^{a,b}_{r})(B^{a,b}_{s+\varepsilon_1}-B^{a,b}_s) \right]\\
&\qquad\qquad +E\left[(B^{a,b}_{s+\varepsilon_1}-B^{a,b}_s)(B^{a,b}_{r+\varepsilon_2}-B^{a,b}_r)
\right]E\left[f(B^{a,b}_{s})f(B^{a,b}_{r})\right]\\
&=E\left[B^{a,b}_{s}(B^{a,b}_{r+\varepsilon_2}-B^{a,b}_r)\right]
E\left[B^{a,b}_{s}(B^{a,b}_{s+\varepsilon_1}-B^{a,b}_s)\right]
E\left[f''(B^{a,b}_{s})f(B^{a,b}_{r})\right]\\
&+E\left[B^{a,b}_{s}(B^{a,b}_{r+\varepsilon_2}-B^{a,b}_r)\right]
E\left[B^{a,b}_{r}(B^{a,b}_{s+\varepsilon_1}-B^{a,b}_s)\right]
E\left[f'(B^{a,b}_{s})f'(B^{a,b}_{r})\right]\\
&\qquad +E\left[B_{r}(B^{a,b}_{r+\varepsilon_2}-B_r)\right]
E\left[B^{a,b}_{s}(B^{a,b}_{s+\varepsilon_1}-B_s)\right]
E\left[f'(B^{a,b}_{s})f'(B^{a,b}_{r})\right]\\
&\qquad +E\left[B^{a,b}_{r}(B^{a,b}_{r+\varepsilon_2}-B^{a,b}_r)\right]
E\left[B^{a,b}_{r}(B^{a,b}_{s+\varepsilon_1}-B^{a,b}_s)\right]
E\left[f(B^{a,b}_{s})f''(B^{a,b}_{r})\right]\\
&\qquad\qquad +E\left[(B^{a,b}_{s+\varepsilon_1}-B^{a,b}_s)(B^{a,b}_{r+\varepsilon_2}-B^{a,b}_r)
\right]E\left[f(B^{a,b}_{s})f(B^{a,b}_{r})\right].
\end{align*}
Denote
\begin{align*}
A_1(s,r,\varepsilon,j):&=\varepsilon_j^{1+b} E\left[B^{a,b}_{s}(B^{a,b}_{r+\varepsilon}-B^{a,b}_r)\right]
E\left[B^{a,b}_{s}(B^{a,b}_{s+\varepsilon}-B^{a,b}_s)\right]\\
&\qquad-\varepsilon^{1+b} E\left[B^{a,b}_{s}(B^{a,b}_{r+\varepsilon_2}-B^{a,b}_r)\right]
E\left[B^{a,b}_{s}(B^{a,b}_{s+\varepsilon_1}-B^{a,b}_s)\right]\\
A_{11}(s,r,\varepsilon,j):&=\varepsilon_j^{1+b}
E\left[B^{a,b}_{r}(B^{a,b}_{r+\varepsilon}-B^{a,b}_r)\right] E\left[B^{a,b}_{s}(B^{a,b}_{s+\varepsilon}-B^{a,b}_s)\right]\\
&\qquad-\varepsilon^{1+b} E\left[B^{a,b}_{r}(B^{a,b}_{r+\varepsilon_2}-B^{a,b}_r)\right]
E\left[B^{a,b}_{s}(B^{a,b}_{s+\varepsilon_1}-B^{a,b}_s)\right]\\
A_{12}(s,r,\varepsilon,j):&=\varepsilon_j^{1+b}
E\left[B^{a,b}_{s}(B^{a,b}_{r+\varepsilon}-B^{a,b}_r)\right]
E\left[B^{a,b}_{r}(B^{a,b}_{s+\varepsilon}-B^{a,b}_s)\right]\\
&\qquad-\varepsilon^{1+b}
E\left[B^{a,b}_{s}(B^{a,b}_{r+\varepsilon_2}-B^{a,b}_r)\right]
E\left[B^{a,b}_{r}(B^{a,b}_{s+\varepsilon_1}-B^{a,b}_s)\right]\\
A_3(s,r,\varepsilon,j):&=\varepsilon_j^{1+b} E\left[B^{a,b}_{r}(B^{a,b}_{r+\varepsilon}-B^{a,b}_r)\right]
E\left[B^{a,b}_{r}(B^{a,b}_{s+\varepsilon}-B^{a,b}_s)\right]\\
&\qquad-\varepsilon^{1+b} E\left[B^{a,b}_{r}(B^{a,b}_{r+\varepsilon_2}-B^{a,b}_r)\right]
E\left[B^{a,b}_{r}(B^{a,b}_{s+\varepsilon_1}-B^{a,b}_s)\right]\\
A_4(s,r,\varepsilon,j):&=\varepsilon_j^{1+b}
E\left[(B^{a,b}_{r+\varepsilon}-B^{a,b}_r) (B^{a,b}_{s+\varepsilon}-B^{a,b}_s)\right]\\
&\qquad-\varepsilon^{1+b}
E\left[(B^{a,b}_{s+\varepsilon_1}-B^{a,b}_s) (B^{a,b}_{r+\varepsilon_2}-B^{a,b}_r)\right]
\end{align*}
with $j=1,2$. It follows that
\begin{align*}
\varepsilon_j^{1+b}&\Phi_{s,r}(1,\varepsilon_i) -\varepsilon_i^{1+b}
\Phi_{s,r}(2,\varepsilon_1,\varepsilon_2)\\
&=A_1(s,r,\varepsilon_i,j)E\left[f''(B^{a,b}_{s})f(B^{a,b}_{r})\right]\\
&\qquad
+(A_{11}(s,r,\varepsilon_i,j)+A_{12}(s,r,\varepsilon_i,j)) E\left[f'(B^{a,b}_{s})f'(B^{a,b}_{r})\right]\\
&\qquad+A_3(s,r,\varepsilon_i,j)E\left[f(B^{a,b}_{s})f''(B^{a,b}_{r})\right]
+A_4(s,r,\varepsilon_i,j)E\left[f(B^{a,b}_{s})f(B^{a,b}_{r})\right]
\end{align*}
with $i,j\in\{1,2\}$ and $i\neq j$. In order to end the proof we claim
that the following convergence
\begin{align}\label{sec4-Con-eq1}
\frac1{\varepsilon_i^{2+2b}\varepsilon_j^{1+b}} \int_{\varepsilon_2}^{t+\varepsilon_1} \int_{\varepsilon_2}^{t+\varepsilon_1} \left\{\varepsilon_j^{1+b}\Phi_{s,r}(1,\varepsilon_i) -\varepsilon_i^{1+b}
\Phi_{s,r}(2,\varepsilon_1,\varepsilon_2)\right\}(sr)^bdrds \longrightarrow 0
\end{align}
with $i,j\in\{1,2\}$ and $i\neq j$, as $\varepsilon_1,\varepsilon_2\to 0$. By symmetry, we only need to show that this holds for $i=1,j=2$. This will be done in three parts.

{\bf Part I}. The following convergence holds:
\begin{align}\label{step1-eq1}
\int_{\varepsilon_2}^{t+\varepsilon_1} \int_{\varepsilon_2}^{t+\varepsilon_1}
\frac{A_4(s,r,\varepsilon_1,2)} {\varepsilon_1^{2+2b}\varepsilon_2^{1+b}} E[f(B^{a,b}_{s})f(B^{a,b}_{r})](sr)^bdrds \longrightarrow 0,
\end{align}
as $\varepsilon_1,\varepsilon_2\to 0$. Clearly, we have
\begin{align*}
|E[(B^{a,b}_{r+\varepsilon_i}-B^{a,b}_r)&(B^{a,b}_{s+\varepsilon_j} -B^{a,b}_s)]|\leq \sqrt{E[(B^{a,b}_{r+\varepsilon_i}-B^{a,b}_r)^2 E(B^{a,b}_{s+\varepsilon_j} -B^{a,b}_s)^2]}\\
&\leq C_{a,b}[(r+\varepsilon_i)(s+\varepsilon_j)]^{a/2} \varepsilon_i^{(1+b)/2}\varepsilon_j^{(1+b)/2}\\
&\leq C_{a,b}[(r+\varepsilon_i)(s+\varepsilon_j)]^{a/2} \frac{\varepsilon_i^{1+b+\gamma} \varepsilon_j^{1+b}
}{(s-r)^{1+b+\gamma}}
\end{align*}
for $0<s-r<\varepsilon_i\wedge \varepsilon_j\leq 1$ and $0<\gamma<-b$, where $i,j\in \{1,2\}$. Combining this with~\eqref{sec3-eq3.8==} (taking $\beta=1$), we get
\begin{align*}
|E[(&B^{a,b}_{r+\varepsilon_1}-B^{a,b}_r)(B^{a,b}_{s+\varepsilon_1} -B^{a,b}_s)]|\\
&\leq C_{a,b}\left(r^a\vee (s+\varepsilon_1)^a\right) \left(\frac{\varepsilon_1^{2+b-\nu}}{(s-r)^{1-\nu}}
1_{\{s-r>\varepsilon_1\}}+ \frac{\varepsilon_1^{2+2b+\gamma}}{(s-r)^{1+b+\gamma}}
1_{\{0<s-r\leq \varepsilon_1\}}\right)\\
&\leq C_{a,b}\left(r^a\vee (s+\varepsilon_1)^a\right) \frac{\varepsilon_1^{2+b-\nu}}{(s-r)^{1-\nu}}
\end{align*}
and
\begin{align*}
|E[(&B^{a,b}_{s+\varepsilon_1} -B^{a,b}_s)(B^{a,b}_{r+\varepsilon_2}-B^{a,b}_r)]|\\
&\leq C_{a,b}\left(r^a\vee (s+\varepsilon_1)^a\right)\left( \frac{\varepsilon_1^{1+b-\nu}\varepsilon_2}{(s-r)^{1-\nu}}
1_{\{s-r>\varepsilon_2\}}+ \frac{\varepsilon_1^{1+b+\gamma}\varepsilon_2^{1+b}}{(s-r)^{1+b+\gamma}}
1_{\{0<s-r\leq \varepsilon_2\}}\right)\\
&\leq C_{a,b}\left(r^a\vee (s+\varepsilon_1)^a\right) \frac{\varepsilon_1^{1-\nu}\varepsilon_2^{1+b}}{(s-r)^{1-\nu}}
\end{align*}
for all $s>r>0$ and $0<\nu<(1+b)\wedge (-b)$ by taking $b+\gamma+\nu\geq 0$. It follows that
\begin{align*}
\frac1{\varepsilon_1^{2+2b}\varepsilon_2^{1+b}} |A_4(s,r,\varepsilon_1,2)|&\leq  \frac{C_{a,b}\left(r^a\vee (s+\varepsilon_1)^a\right)}{(s-r)^{1-\nu}} \varepsilon_1^{-b-\nu}\longrightarrow 0
\end{align*}
for all $s>r>0$ and $\nu<-b$, as $\varepsilon_1,\varepsilon_2\to 0$.

On the other hand, from the above proof we have also
\begin{align*}
\frac1{\varepsilon_1^{2+2b}\varepsilon_2^{1+b}} |A_4(s,r,\varepsilon_1,2)|&\leq  \frac1{\varepsilon_1^{2+2b}} |E[(B^{a,b}_{r+\varepsilon_1}-B^{a,b}_r)(B^{a,b}_{s+\varepsilon_1} -B^{a,b}_s)]|\\
&\qquad+\frac1{\varepsilon_1^{1+b}\varepsilon_2^{1+b}}
|E[(B^{a,b}_{s+\varepsilon_1} -B^{a,b}_s)(B^{a,b}_{r+\varepsilon_2}-B^{a,b}_r)]|\\
&\leq C_{a,b}\frac{r^a\vee (s+\varepsilon_1)^a}{(s-r)^{1+b}}
\end{align*}
for all $s>r>0$ and $\varepsilon_1,\varepsilon_2>0$, and
$$
\int_{\varepsilon_2}^{t+\varepsilon_1} \int_{\varepsilon_2}^{t+\varepsilon_1} \frac{r^a}{(s-r)^{1+b}} |E[f(B^{a,b}_{s})f(B^{a,b}_{r})]|(sr)^bdrds\leq C_{a,b}\|f\|^2_{\mathscr H}
$$
for any $0<\varepsilon_1,\varepsilon_2<1$. Using Lebesgue's dominated convergence theorem we can deduce the convergence~\eqref{step1-eq1}.

{\bf Part II}. The following convergence holds:
\begin{align}\label{step2-eq1}
\int_{\varepsilon_2}^{t+\varepsilon_1} \int_{\varepsilon_2}^{t+\varepsilon_1}
\frac{A_{11}(s,r,\varepsilon_1,2)+A_{12}(s,r,\varepsilon_1,2)}{ \varepsilon_1^{2+2b}\varepsilon_2^{1+b}} E[f'(B^{a,b}_{s})f'(B^{a,b}_{r})](sr)^bdrds\to 0,
\end{align}
as $\varepsilon_1,\varepsilon_2\to 0$. By Lemma~\ref{lem3.4} we have
\begin{align*}
\frac1{\varepsilon_1^{2+2b}\varepsilon_2^{1+b}}& \left(|A_{11} (s,r,\varepsilon_1,2)|+|A_{12}(s,r,\varepsilon_1,2)|\right)\\
&\leq |E[B^{a,b}_{s}(B^{a,b}_{s+\varepsilon_1}-B^{a,b}_s)]|\\
&\qquad\qquad\cdot\left(\varepsilon_2^{1+b} |E[B^{a,b}_{r}(B^{a,b}_{r+\varepsilon_1}-B^{a,b}_r)]| +\varepsilon_1^{1+b} |E[B^{a,b}_{r}(B^{a,b}_{r+\varepsilon_2}-B^{a,b}_r)]|\right)\\
&\quad +|E[B^{a,b}_{r}(B^{a,b}_{s+\varepsilon_1}-B^{a,b}_s)]|\\
&\qquad\qquad\cdot\left(
\varepsilon_2^{1+b}|E[B^{a,b}_{s}(B^{a,b}_{r+\varepsilon_1}-B^{a,b}_r)]| +\varepsilon_1^{1+b}
|E[B^{a,b}_{s}(B^{a,b}_{r+\varepsilon_2}-B^{a,b}_r)]|\right)\\
&\leq C_{a,b}s^ar^a
\end{align*}
for all $s>r>0$, and moreover, by Lemma~\ref{lem3.1} and Lemma~\ref{lem3.5} we have
\begin{align*}
&\int_{\varepsilon_2}^{t+\varepsilon_1} \int_{\varepsilon_2}^{t+\varepsilon_1}
s^ar^a|E[f'(B^{a,b}_{s})f'(B^{a,b}_{r})]|(sr)^bdrds\\
&\;\;\leq C_{a,b}\int_{\varepsilon_2}^{t+\varepsilon_1} \int_{\varepsilon_2}^{t+\varepsilon_1}
\left(\frac{(rs)^{\frac12(1+a+b)}}{\rho^2_{s,r}} +\frac{\mu_{s,r}}{\rho^2_{s,r}}\right)\left(E|f(B^{a,b}_s)|^2
+E|f(B^{a,b}_r)|^2\right)(sr)^{a+b}drds\\
&\;\;=C_{a,b}\int_{\varepsilon_2}^{t+\varepsilon_1} \int_{\varepsilon_2}^{t+\varepsilon_1}
\left(\frac{(rs)^{\frac12(1+a+b)}}{\rho^2_{s,r}} +\frac{\mu_{s,r}}{\rho^2_{s,r}}\right)E|f(B^{a,b}_s)|^2(sr)^{a+b}drds \leq C_{a,b}\|f\|_{\mathscr H}^2
\end{align*}
for all $\varepsilon_1,\varepsilon_2>0$

On the other hand, we have
\begin{align*}
E[B^{a,b}_{r}&(B^{a,b}_{r+\varepsilon}-B^{a,b}_r)]\\
&=\frac1{2{\mathbb B}(1+a,1+b)}
\left(\int_0^ru^a(r+\varepsilon-u)^bdu-{\mathbb B}(1+a,1+b)r^{1+a+b}\right)\\
&=\frac1{2{\mathbb B}(1+a,1+b)} \int_0^{\frac{r}{r+\varepsilon}}x^a(1-x)^bdx \left\{(r+\varepsilon)^{1+a+b}- r^{1+a+b}\right\}\\
&\qquad-\frac{r^{1+a+b}}{2{\mathbb B}(1+a,1+b)} \int_{\frac{r}{r+\varepsilon}}^1x^a(1-x)^bdx\\
&\equiv \frac1{2{\mathbb B}(1+a,1+b)} \left\{\Lambda_1(r,\varepsilon)-\Lambda_2(r,\varepsilon)\right\}
\end{align*}
for $r>0$ and $\varepsilon>0$, and
\begin{align*}
\Bigl|\varepsilon_2^{1+b}
E[&B^{a,b}_{r}(B^{a,b}_{r+\varepsilon_1}-B^{a,b}_r)]
-\varepsilon_1^{1+b}E[B^{a,b}_{r}(B^{a,b}_{r+\varepsilon_2} -B^{a,b}_r)]\Bigr|\\
&\leq \varepsilon_2^{1+b}|\Lambda_1(r,\varepsilon_1)| +\varepsilon_1^{1+b}|\Lambda_1(r,\varepsilon_2)| +\left|\varepsilon_2^{1+b}\Lambda_2(r,\varepsilon_1) -\varepsilon_1^{1+b}\Lambda_2(r,\varepsilon_2)\right|
\end{align*}
for $r>0$ and $\varepsilon_1>\varepsilon_2>0$. Notice that for $r>0$ and $\varepsilon>0$,
\begin{align*}
|\Lambda_1(r,\varepsilon)|&= \int_0^{\frac{r}{r+\varepsilon}}x^a(1-x)^bdx \left|(r+\varepsilon)^{1+a+b}- r^{1+a+b}\right|\\
&\leq C_{a,b}\left(\frac{r}{r+\varepsilon}\right)^{1+a} \left|(r+\varepsilon)^{1+a+b}- r^{1+a+b}\right|\\
&\leq C_{a,b}\left(\frac{r}{r+\varepsilon}\right)^{1+a}
(r+\varepsilon)^{a+b}\varepsilon\leq C_{a,b}r^{a+b}\varepsilon
\end{align*}
by~\eqref{lem3.4-eq7} and the fact
\begin{equation}\label{step2-eq2-1}
y^\alpha-x^\alpha\leq C_{\alpha}y^{\alpha-1}(y-x),\quad y>x>0,\quad \alpha>0,
\end{equation}
and
\begin{align*}
|\varepsilon_2^{1+b}&\Lambda_2(r,\varepsilon_1) -\varepsilon_1^{1+b}\Lambda_2(r,\varepsilon_2)|\\
&=C_{a,b}r^{1+a+b}\left|\varepsilon_2^{1+b} \int_{\frac{r}{r+\varepsilon_1}}^1x^a(1-x)^bdx-
\varepsilon_1^{1+b} \int_{\frac{r}{r+\varepsilon_2}}^1x^a(1-x)^bdx\right|\\
&=C_{a,b}r^{1+a+b}(\varepsilon_1\varepsilon_2)^{1+b} \left|\frac1{\varepsilon_1^{1+b}} \int_{\frac{r}{r+\varepsilon_1}}^1x^a(1-x)^bdx-
\frac1{\varepsilon_2^{1+b}} \int_{\frac{r}{r+\varepsilon_2}}^1x^a(1-x)^bdx\right|\\
&\equiv C_{a,b}r^{1+a+b}(\varepsilon_1\varepsilon_2)^{1+b}
\Lambda_3(r,\varepsilon_1,\varepsilon_2)
\end{align*}
for $r>0$ and $\varepsilon_1>\varepsilon_2>0$. We get
\begin{align*}
\frac1{\varepsilon_1^{2+2b}\varepsilon_2^{1+b}}& |A_{11}(s,r,\varepsilon_1,2)|\leq \frac1{\varepsilon_1^{2+2b}\varepsilon_2^{1+b}} |E[B^{a,b}_{s}(B^{a,b}_{s+\varepsilon_1}-B^{a,b}_s)]|\\
&\qquad\cdot\left(\varepsilon_2^{1+b}
E[B^{a,b}_{r}(B^{a,b}_{r+\varepsilon_1}-B^{a,b}_r)] -\varepsilon_1^{1+b} E[B^{a,b}_{r}(B^{a,b}_{r+\varepsilon_2}-B^{a,b}_r)]\right)\\
&\leq C_{a,b}s^ar^{a+b}\left(\varepsilon_1^{-b}+\varepsilon_2^{-b}+
r\Lambda_3(r,\varepsilon_1,\varepsilon_2)\right)\longrightarrow 0
\quad (\varepsilon_1,\varepsilon_2\to 0)
\end{align*}
for all $s,r>0$ and $-1<b<0$ because
$$
\Lambda_3(r,\varepsilon_1,\varepsilon_2)\longrightarrow 0
$$
as $0<\varepsilon_2<\varepsilon_1\to 0$ by the convergence
$$
\lim_{x\uparrow 1}\frac1{(1-x)^{1+b}}\int_x^1u^a(1-u)^bdu=\frac1{1+b},
$$
and similarly we also have
\begin{align*}
\frac1{\varepsilon_1^{2+2b}\varepsilon_2^{1+b}}&|A_{12} (s,r,\varepsilon_1,2)| \longrightarrow 0\quad (\varepsilon_1,\varepsilon_2\to 0)
\end{align*}
for all $s,r>0$, and the convergence~\eqref{step2-eq1} follow from Lebesgue's dominated convergence theorem again.

{\bf Part III}. The following convergence holds:
\begin{equation}\label{step3-eq1}
\begin{split}
\frac1{\varepsilon_1^{2+2b}\varepsilon_2^{1+b}} \int_{\varepsilon_2}^{t+\varepsilon_1} &\int_{\varepsilon_2}^{t+\varepsilon_1} \Bigl(|A_1(s,r,\varepsilon_1,2)E[f''(B^{a,b}_{s})f(B^{a,b}_{r})]|\\
&\qquad +|A_3(s,r,\varepsilon_1,2)E[f(B^{a,b}_{s})f''(B^{a,b}_{r})]|
\Bigr)(sr)^bdrds \longrightarrow 0,
\end{split}
\end{equation}
as $\varepsilon_1,\varepsilon_2\to 0$. Clearly, for all $s>r>0$ and all $\varepsilon_1,\varepsilon_2>0$ we have
\begin{align*}
\frac1{\varepsilon_1^{2+2b}\varepsilon_2^{1+b}} (|&A_1(s,r,\varepsilon_1,2)|+|A_3(s,r,\varepsilon_1,2)|)
\leq \frac1{\varepsilon_1^{2+2b}\varepsilon_2^{1+b}} |E[B^{a,b}_{s}(B^{a,b}_{s+\varepsilon_1}-B^{a,b}_s)]|\\
&\qquad\cdot\left(\varepsilon_2^{1+b}| E[B^{a,b}_{s}(B^{a,b}_{r+\varepsilon_1}-B^{a,b}_r)]|+ \varepsilon_1^{1+b}|E[B^{a,b}_{s}(B^{a,b}_{r+\varepsilon_2}-B^{a,b}_r)] |\right)\\
&+\frac1{\varepsilon_1^{2+2b}\varepsilon_2^{1+b}} |E[B^{a,b}_{r}(B^{a,b}_{s+\varepsilon_1}-B^{a,b}_s)]|\\
&\qquad \cdot\left(
\varepsilon_2^{1+b} |E[B^{a,b}_{r}(B^{a,b}_{r+\varepsilon_1}-B^{a,b}_r)]|
+\varepsilon_1^{1+b} |E[B^{a,b}_{r}(B^{a,b}_{r+\varepsilon_2}-B^{a,b}_r)]|\right)\\
&\leq C_{a,b}s^ar^a
\end{align*}
by Lemma~\ref{lem3.4}, and
\begin{align*}
\int_{\varepsilon_2}^{t+\varepsilon_1} &\int_{\varepsilon_2}^{t+\varepsilon_1} s^ar^a\Bigl(|E[f''(B^{a,b}_{s})f(B^{a,b}_{r})]| +|E[f(B^{a,b}_{s})f''(B^{a,b}_{r})]|
\Bigr)(sr)^bdrds\\
&\leq C_{a,b}\int_{\varepsilon_2}^{t+\varepsilon_1} \int_{\varepsilon_2}^{t+\varepsilon_1} s^ar^a\frac{r^{1+a+b}+s^{1+a+b}}{\rho^2_{s,r}}( E|f(B^{a,b}_s)|^2+E|f(B^{a,b}_r)|^2)(sr)^bdrds\\
&\leq C_{a,b} \int_{\varepsilon_2}^{t+\varepsilon_1} s^{a+b}E|f(B^{a,b}_s)|^2ds\leq C_{a,b}\|f\|_{{\mathscr H}}^2
\end{align*}
by Lemma~\ref{lem3.5} and Lemma~\ref{lem3.1}.

On the other hand, by the fact~\eqref{step2-eq2-1} we have
\begin{equation}\label{sec4-eq4.2200-2}
\begin{split}
\int_r^{r+\varepsilon}u^a(s-u)^bdu&\leq \frac1{1+b}\left(r^a\vee (r+\varepsilon)^a\right) \left((s-r)^{1+b}-(s-r-\varepsilon)^{1+b}\right)\\
&\leq \frac1{1+b}\left(r^a\vee s^a\right)(s-r)^{1+b-\beta}\varepsilon^\beta
\end{split}
\end{equation}
for $r+\varepsilon<s$ and $1+b<\beta<1$, which implies that
\begin{equation}\label{sec4-eq4.2200}
\begin{split}
|E[B^{a,b}_s&(B^{a,b}_{r+\varepsilon}-B^{a,b}_r)]| 1_{\{s-r>\varepsilon\}}=C_{a,b}\left(\int_0^{r+\varepsilon}u^a[(s-u)^b +({r+\varepsilon}-u)^b]du\right.\\
&\hspace{3cm}\left.-
\int_0^ru^a[(s-u)^b+(r-u)^b]du\right)1_{\{s-r>\varepsilon\}}\\
&=C_{a,b}\left(\int_r^{r+\varepsilon}u^a(s-u)^bdu +(r+\varepsilon)^{1+a+b}-r^{1+a+b} \right)1_{\{s-r>\varepsilon\}}\\
&\leq C_{a,b}\varepsilon^\beta \left((r^a\vee s^a)(s-r)^{1+b-\beta} +(r+\varepsilon)^{1+a+b-\beta}\right) 1_{\{s-r>\varepsilon\}}
\end{split}
\end{equation}
for all $1+b<\beta<1$. Moreover, by Lemma~\ref{lem3.4} we have
\begin{align*}
|E[B^{a,b}_s(B^{a,b}_{r+\varepsilon}-B^{a,b}_r)]| 1_{\{s-r<\varepsilon\}}&\leq C_{a,b}s^a\varepsilon^{1+b}1_{\{0<s-r<\varepsilon\}}\\
&\leq C_{a,b}s^a \varepsilon^{\beta}(s-r)^{1+b-\beta}1_{\{0<s-r<\varepsilon\}}
\end{align*}
for all $1+b<\beta<1$. It follows from Lemma~\ref{lem3.4} that
\begin{align*}
\frac1{\varepsilon_1^{2+2b}\varepsilon_2^{1+b}} &|A_1(s,r,\varepsilon_1,2)|
\leq \frac1{\varepsilon_1^{2+2b}\varepsilon_2^{1+b}} |E[B^{a,b}_{s}(B^{a,b}_{s+\varepsilon_1}-B^{a,b}_s)]|\\
&\qquad \cdot\left(\varepsilon_2^{1+b}| E[B^{a,b}_{s}(B^{a,b}_{r+\varepsilon_1}-B^{a,b}_r)]|+ \varepsilon_1^{1+b}|E[B^{a,b}_{s}(B^{a,b}_{r+\varepsilon_2}-B^{a,b}_r)] |\right)\\
&\leq C_{a,b}s^a
\left((r^a\vee s^a)(s-r)^{1+b-\beta} +(r+\varepsilon_1)^{1+a+b-\beta}\right) \varepsilon_1^{\beta-1-b}
1_{\{s-r>\varepsilon_1\}}\\
&\quad+C_{a,b}s^a
\left((r^a\vee s^a)(s-r)^{1+b-\beta} +(r+\varepsilon_2)^{1+a+b-\beta}\right) \varepsilon_2^{\beta-1-b}
1_{\{s-r>\varepsilon_2\}}\\
&\quad+C_{a,b}s^{2a}(s-r)^{1+b-\beta}\varepsilon^{\beta}_1
1_{\{0<s-r<\varepsilon_1\}} +C_{a,b}s^{2a}(s-r)^{1+b-\beta}\varepsilon^{\beta}_2
1_{\{0<s-r<\varepsilon_2\}}\\
&\longrightarrow 0\quad (\varepsilon_1,\varepsilon_2\to 0)
\end{align*}
for all $s>r>0$. Similarly, for all $s>r>0$ and $1+b<\beta<1$ we have also
\begin{align*}
\frac1{\varepsilon_1^{2+2b}\varepsilon_2^{1+b}} &|A_3(s,r,\varepsilon_1,2)|\leq \frac1{\varepsilon_1^{2+2b}\varepsilon_2^{1+b}} |E[B^{a,b}_{r}(B^{a,b}_{s+\varepsilon_1}-B^{a,b}_s)]|\\
&\qquad \cdot\left(
\varepsilon_2^{1+b} |E[B^{a,b}_{r}(B^{a,b}_{r+\varepsilon_1}-B^{a,b}_r)]|
+\varepsilon_1^{1+b} |E[B^{a,b}_{r}(B^{a,b}_{r+\varepsilon_2}-B^{a,b}_r)]|\right)\\
&\leq C_{a,b}(s-r)^{b}s^{a-\gamma+1}r^a\varepsilon_1^{\gamma-1-b} \longrightarrow 0\quad (\varepsilon_1,\varepsilon_2\to 0)
\end{align*}
by Lemma~\ref{lem3.4} and using the estimate
\begin{align*} |E[B^{a,b}_{r}(B^{a,b}_{s+\varepsilon_1}-B^{a,b}_s)]&=
\frac1{2{\mathbb B}(1+a,1+b)} \int_0^ru^a[(s-u)^b -({s+\varepsilon_1}-u)^b]du \\
&=\frac1{2{\mathbb B}(1+a,1+b)} \int_0^ru^a\frac{({s+\varepsilon_1}-u)^{-b}-(s-u)^{-b}}
{({s+\varepsilon_1}-u)^{-b}(s-u)^{-b}}du\\
&\leq \frac1{2{\mathbb B}(1+a,1+b)}\varepsilon_1^\gamma
\int_0^r\frac{u^a}
{({s+\varepsilon_1}-u)^{\gamma}(s-u)^{-b}}du\\
&\leq \frac1{2{\mathbb B}(1+a,1+b)}\frac{\varepsilon_1^\gamma}{(s-r)^{-b}}
\int_0^s\frac{u^a}{(s-u)^{\gamma}}du\\ &=C_{a,b}(s-r)^{b}s^{a-\gamma+1}\varepsilon_1^\gamma
\end{align*}
for all $s>r>0,\;(-b)\vee (1+b)<\gamma<1$. Consequently, Lebesgue's dominated convergence theorem implies that the convergence~\eqref{step3-eq1} holds again.

Thus, we have established the convergence~\eqref{sec4-Con-eq1} for $i=1,j=2$ and the theorem follows.
\end{proof}

\begin{corollary}\label{cor5.2}
Let $f,f_1,f_2,\ldots\in {\mathscr H}$. If $f_n\to f$ in ${\mathscr
H}$, as $n$ tends to infinity, then we have
$$
[f_n(B^{a,b}),B^{a,b}]^{(a,b)}_t\longrightarrow [f(B^{a,b}),B^{a,b}]^{(a,b)}_t
$$
in $L^2$ as $n\to \infty$.
\end{corollary}
\begin{proof}
The corollary follows from
$$
E\left|[f_n(B^{a,b}),B^{a,b}]^{(a,b)}_t -[f(B^{a,b}),B^{a,b}]^{(a,b)}_t\right|^2\leq
C_{a,b}\|f_n-f\|_{\mathscr H}^2\to 0,
$$
as $n$ tends to infinity.
\end{proof}

\section{The generalized Bouleau-Yor identity with $b<0$}\label{sec7}

In this section, we study the Bouleau-Yor identity. It is known that the quadratic covariation $[f(B),B]$ of Brownian motion $B$ can be characterized
as
$$
[f(B),B]_t=-\int_{\mathbb R}f(x){\mathscr L}^{B}(dx,t),
$$
where $f$ is locally square integrable, ${\mathscr L}^{B}(x,t)$
is the local time of Brownian motion and the quadratic covariation $\bigl[f(B),B\bigr]$ (see Russo-Vallois~\cite{Russo2}) is defined by
$$
\bigl[f(B),B\bigr]:=\lim_{\varepsilon\downarrow 0}\frac{1}{\varepsilon} \int_0^{t}\left\{ f(B_{s+\varepsilon})
-f(B_s)\right\}(B_{s+\varepsilon}-B_s)ds,
$$
provided the limit exists uniformly in probability. This is called the Bouleau-Yor identity. More works for this can be found in
Bouleau-Yor~\cite{Bouleau}, Eisenbaum~\cite{Eisen1}, F\"ollmer {\it
et al}~\cite{Follmer}, Feng--Zhao~\cite{Feng,Feng3},
Peskir~\cite{Peskir1}, Rogers--Walsh~\cite{Rogers2},
Yan et al~\cite{Yan4,Yan1,Yan2} and the references therein. However, the Bouleau-Yor identity is not true for weighted-fBm with $b<0$ because
$$
[B^{a,b},B^{a,b}]_t=+\infty
$$
for all $t>0$ and $-1<b<0$. In this section, we shall obtain a generalized Bouleau-Yor identity based on the generalized quadratic covariation defined in Section~\ref{sec6}

Recall that for any closed interval $I\subset {\mathbb R}_{+}$ and for any $x\in {\mathbb R}$, the local time $L(x,I)$ of $B^{a,b}$ is defined as
the density of the occupation measure $\mu_I$ defined by
$$
\mu_I(A)=\int_I1_A(B^{a,b}_s)ds
$$
It can be shown (see Geman and Horowitz~\cite{Geman}, Theorem 6.4) that the following occupation density formula holds:
$$
\int_Ig(B^{a,b}_s,s)ds=\int_{\mathbb R}dx\int_Ig(x,s) L(x,ds)
$$
for every Borel function $g(x,t)\geq 0$ on $I\times {\mathbb R}$. Thus, Lemma~\ref{lem3.0} in Section~\ref{sec3} and Theorem 21.9 in Geman-Horowitz~\cite{Geman} together imply that the following result holds.

\begin{corollary}
Let $a>-1,|b|<1,|b|<1+a$ and let $L(x,t):=L(x, [0,t])$ be the local time of $B^{a,b}$ at $x$. Then $L\in L^2(\lambda\times P)$ for all $t\geq 0$ and $(x,t)\mapsto L(x,t)$ is jointly continuous if and only if $a+b<3$, where $\lambda$ denotes Lebesgue measure. Moreover, the occupation formula
\begin{equation}\label{sec2-1-eq1}
\int_0^t\psi(B^{a,b}_s,s)ds=\int_{\mathbb R}dx\int_0^t\psi(x,s) L(x,ds)
\end{equation}
holds for every continuous and bounded function $\psi(x,t):{\mathbb
R}\times {\mathbb R}_{+}\rightarrow {\mathbb R}$ and any $t\geqslant
0$.
\end{corollary}
Define the weighted local time ${\mathscr L}^{a,b}$ by
\begin{align*}
{\mathscr L}^{a,b}(x,t)&=(1+a+b)\int_0^ts^{a+b}d_sL(s,x)\\
&\equiv (1+a+b)\int_0^t\delta(B^{a,b}_s-x)s^{a+b}ds
\end{align*}
for $t\geq 0$ and $x\in {\mathbb R}$, where $\delta$ is the Dirac delta function. In this section, we define the integral
\begin{equation}\label{sec7-eq7.1}
\int_{\mathbb R}f(x){\mathscr L}^{a,b}(dx,t),
\end{equation}
for $b<0$, where $f$ is a Borel function. We shall use the generalized quadratic covariation to study it and obtain the following generalized Bouleau-Yor identity:
\begin{equation}\label{Bouleau-Yor}
[f(B^{a,b}),B^{a,b}]_t^{(a,b)}=-\kappa_{a,b}\int_{\mathbb R}f(x){\mathscr L}^{a,b}(dx,t)
\end{equation}
for all $f\in {\mathscr H}$, $-1<b<0$, $a>-1$, $-1<a+b<3$ and $t\geq 0$. We first give an extension of It\^{o} formula stated as follows.
\begin{theorem}\label{th7.1}
Let $a>-1$, $-1<b<0$, $|b|<1+a$ and let $f\in {\mathscr H}$ be a left continuous function with right limits. If $F$ is an absolutely continuous function with $F'=f$ and
\begin{equation}\label{sec7-Ito-con1}
\max\left\{|F(x)|,|f(x)|\right\}\leq C e^{\beta x^2},
\end{equation}
where $C$ and $\beta$ are some positive constants with
$\beta<\frac14T^{-(1+a+b)}$, then the It\^o type formula
\begin{equation}\label{sec7-eq7.3}
F(B^{a,b}_t)=F(0)+\int_0^tf(B^{a,b}_s)dB^{a,b}_s+\frac12(\kappa_{a,b})^{-1}
[f(B^{a,b}),B^{a,b}]^{(a,b)}_t
\end{equation}
holds for all $t\in [0,T]$.
\end{theorem}

Clearly, this is an analogue of F\"ollmer-Protter-Shiryayev's
formula F\"ollmer (see, for example, F\"ollmer {\it et al}~\cite{Follmer}). It is an improvement in terms of the
hypothesis on $f$ and it is also quite interesting itself.

Based on the localization argument and smooth approximation one can prove Theorem~\ref{th7.1}. The localization is that one can localize the domain ${\rm Dom}(\delta^{a,b})$ of the operator $\delta^{a,b}$ (see Nualart~\cite{Nualart1}). Suppose that
$\{(\Omega_n, u^{(n)}), n\geq 1\}\subset {\mathscr F}\times {\rm
Dom}(\delta^{a,b})$ is a localizing sequence for $u$, i.e., the
sequence $\{(\Omega_n, u^{(n)}), n\geq 1\}$ satisfies
\begin{itemize}
\item [(i)] $\Omega_n\uparrow \Omega$, a.s.;
\item [(ii)] $u=u^{(n)}$ a.s. on $\Omega_n$.
\end{itemize}
If $\delta^{a,b}(u^{(n)})= \delta^{a,b}(u^{(m)})$ a.s. on $\Omega_n$ for
all $m\geq n$, then, the divergence $\delta^{a,b}$ is the random
variable determined by the conditions
$$
\delta^{a,b}(u)|_{\Omega_n}=\delta^{a,b} (u^{(n)})|_{\Omega_n}\qquad
{\rm { for\;\; all\;\;}} n\geq 1,
$$
which may depend on the localizing sequence.
\begin{lemma}[Nualart~\cite{Nualart1}]\label{complete2.1} Let
$\{v^{(n)}\}$ be a sequence such that $v^{(n)}\to v$ in $L^2$, as $n\to \infty$ and let
$$
\delta^{a,b}(v^{(n)})=\int_0^Tv_s^{(n)}dB^{a,b}_s,\qquad n\geq 1
$$
exist in $L^2$. If $\delta^{a,b}(v^{(n)})\to G$ in $L^2$, then
$\delta^{a,b}(v)=\int_0^Tv_sdB^{a,b}_s$ exists in $L^2$ and equals to $G$.
\end{lemma}

\begin{proof}[Proof of Theorem~\ref{th7.1}]
If $F\in C^2({\mathbb R})$, then this is It\^o's formula since
$$
[f(B^{a,b}),B^{a,b}]^{(a,b)}_t=(1+a+b) \kappa_{a,b} \int_0^tf'(B^{a,b}_s)s^{a+b}ds.
$$
by~\eqref{sec6-eq6.2}. The assumption $F\in C^2({\mathbb R})$ is not correct. By a localization argument we may
assume that the function $f$ is uniformly bounded. In fact, for any
$k\geq 0$ we may consider the set
$$
\Omega_k=\left\{\sup_{0\leq t\leq T}|B^{a,b}_t|<k\right\}
$$
and let $f^{[k]}$ be a measurable function such that $f^{[k]}=f$ on
$[-k,k]$ and such that $f^{[k]}$ vanishes outside. Then $f^{[k]}\in
{\mathscr H}$ and uniformly bounded. Set
$\frac{d}{dx}F^{[k]}=f^{[k]}$ and $F^{[k]}=F$ on $[-k,k]$. If the
theorem is true for all uniformly bounded functions on ${\mathscr
H}$, then we get the desired formula
$$
F^{[k]}(B^{a,b}_t)=F^{[k]}(0)+\int_0^t
f^{[k]}(B^{a,b}_s)dB^{a,b}_s+\frac12(\kappa_{a,b})^{-1}
\bigl[f^{[k]}(B^{a,b}),B^{a,b}\bigr]^{(a,b)}_t
$$
on the set $\Omega_k$. Letting $k$ tend to infinity we deduce the
It\^o formula~\eqref{sec7-eq7.3} for all $f\in {\mathscr H}$ being
left continuous and locally bounded.

Let now $F'=f\in {\mathscr H}$ be uniformly bounded such that the conditions in the theorem hold. Define the sequence of smooth functions
$$
F_n(x):=\int_{\mathbb R}F(x-{y})\zeta_n(y)dy,\quad x\in {\mathbb R},
$$
where the mollifiers $\zeta_n,n\geq 1$ are given by~\eqref{sec7-eq7.4}. Then $F_n\in C^\infty({\mathbb R})$ and $F_n$ satisfies the condition~\eqref{sec2-Ito-con1}, and the It\^{o} formula
\begin{equation}\label{sec7-eq7.5}
F_n(B^{a,b}_t)=F_n(0)+\int_0^tf_n(B^{a,b}_s)dB^{a,b}_s+
\frac12(1+a+b)\int_0^tf_n'(B^{a,b}_s)s^{a+b}ds
\end{equation}
holds for all $n\geq 1$, where $f_n=F_n'$. We also have
$$
F_n \longrightarrow F,\quad f_n \longrightarrow f,
$$
uniformly in $\mathbb R$, as $n$ tends to infinity, and moreover,  $\{f_n\}\subset {\mathscr H}$, $f_n\to f$ in ${\mathscr H}$, as $n$ tends to infinity. It follows that
\begin{align*}
(1+a+b)\int_0^tf_n'(B^{a,b}_s)s^{a+b}ds &=(\kappa_{a,b})^{-1}\bigl[f_n(B^{a,b}),B^{a,b}\bigr]^{(a,b)}_t\\
&\longrightarrow (\kappa_{a,b})^{-1}\bigl[f(B^{a,b}),B^{a,b}\bigr]^{(a,b)}_t
\end{align*}
in $L^2(\Omega)$ by Corollary~\ref{cor5.2}, as $n$ tends to infinity. It follows that
\begin{align*}
\int_0^tf_n(B^{a,b}_s)dB^{a,b}_s&=F_n(B^{a,b}_t)-F_n(0)-
\frac12(\kappa_{a,b})^{-1}\bigl[f_n(B^{a,b}),B^{a,b}\bigr]^{(a,b)}_t\\
&\longrightarrow F(B^{a,b}_t)-F(0)-\frac12(\kappa_{a,b})^{-1} \bigl[f(B^{a,b}),B^{a,b}\bigr]^{(a,b)}_t
\end{align*}
in $L^2(\Omega)$, as $n$ tends to infinity. This completes the proof since
the integral is closed in $L^2(\Omega)$.
\end{proof}

Now, let us study the integral~\eqref{sec7-eq7.1} for $-1<b<0$.
\begin{lemma}
Let $-1<b<0$, $a>-1$ and $-1<a+b<3$. For any $f_\triangle=\sum_jf_j1_{(a_{j-1},a_j]}\in {\mathscr H}$, we
define
$$
\int_{\mathbb R}f_\triangle(x){\mathscr
L}^{a,b}(dx,t):=\sum_jf_j\left[{\mathscr L}^{a,b}(a_j,t)-{\mathscr
L}^{a,b}(a_{j-1},t)\right].
$$
Then the integral is well-defined and
\begin{equation}\label{sec6-eq6.4}
\kappa_{a,b}\int_{\mathbb R}f_{\Delta}(x)\mathscr{L}^{a,b}(dx,t)=
-\bigl[f_\triangle(B^{a,b}),B^{a,b}\bigr]^{(a,b)}_t
\end{equation}
almost surely, for all $t\in [0,T]$.
\end{lemma}
\begin{proof}
For the function $f_\triangle(x)=1_{(a,b]}(x)$ we define the sequence of smooth functions $f_n,\;n=1,2,\ldots$ by
\begin{align}
f_n(x)&=\int_{\mathbb
R}f_\triangle(x-y)\zeta_n(y)dy=\int_a^b\zeta_n(x-u)du
\end{align}
for all $x\in \mathbb R$, where $\zeta_n,n\geq 1$ are the so-called
mollifiers given in~\eqref{sec7-eq7.4}. Then $\{f_n\}\subset
C^{\infty}({\mathbb R})\cap {\mathscr H}$ and $f_n$ converges to $f_\triangle$ in ${\mathscr H}$, as $n$ tends to infinity. It follows from  the occupation formula that
\begin{align*}
\bigl[f_n(B^{a,b}),B^{a,b}\bigr]^{(a,b)}_t& =\kappa_{a,b}\int_0^tf'_n(B^{a,b}_s)ds^{1+a+b}\\
&=\kappa_{a,b}\int_{\mathbb R}f_n'(x){\mathscr L}^H(x,t)dx=\int_{\mathbb
R}\left(\int_a^b\zeta_n'(x-u)du\right){\mathscr L}^{a,b}(x,t)dx\\
&=-\kappa_{a,b}\int_{\mathbb R}{\mathscr
L}^{a,b}(x,t)\left(\zeta_n(x-b)-\zeta_n(x-a)\right)dx\\
&=\kappa_{a,b}\int_{\mathbb R}{\mathscr L}^{a,b}(x,t)\zeta_n(x-a)dx
-\int_{\mathbb R}{\mathscr L}^{a,b}(x,t)\zeta_n(x-b)dx\\
&\longrightarrow \kappa_{a,b}\left({\mathscr L}^{a,b}(a,t)-{\mathscr L}^{a,b}(b,t)\right)
\end{align*}
almost surely, as $n\to \infty$, by the continuity of $x\mapsto
{\mathscr L}^{a,b}(x,t)$. On the other hand, Corollary~\ref{cor5.2}
implies that there exists a subsequence $\{f_{n_k}\}$ such that
$$
\bigl[f_{n_k}(B^{a,b}),B^{a,b}\bigr]^{(a,b)}_t\longrightarrow \bigl[1_{(a,b]}(B^{a,b}),B^{a,b}\bigr]^{(a,b)}_t
$$
for all $t\in [0,T]$, almost surely, as $k\to \infty$. Then we have
$$
\bigl[1_{(a,b]}(B^{a,b}),B^{a,b}\bigr]^{(a,b)}_t =\kappa_{a,b}\left({\mathscr L}^{a,b}(a,t)-{\mathscr L}^{a,b}(b,t)
\right)
$$
for all $t\in [0,T]$, almost surely. Thus, the identity
$$
\kappa_{a,b}\sum_jf_j\bigl[{\mathscr L}^{a,b}(a_j,t)-{\mathscr
L}^{a,b}(a_{j-1},t)\bigr]=
-\bigl[f_\triangle(B^{a,b}),B^{a,b}\bigr]^{(a,b)}_t
$$
follows from the linearity property, and the lemma follows.
\end{proof}

As a direct consequence of the above lemma we can show
that
\begin{equation} \lim_{n\to \infty}\int_{\mathbb
R}f_{\triangle,n}(x){\mathscr L}^{a,b}(x,t)dx =\lim_{n\to
\infty}\int_{\mathbb R}g_{\triangle,n}(x){\mathscr L}^{a,b}(x,t)dx
\end{equation}
in $L^2(\Omega)$ if
$$
\lim_{n\to \infty}f_{\triangle,n}(x)=\lim_{n\to
\infty}g_{\triangle,n}(x)=f(x)
$$
in ${\mathscr H}$, where $\{f_{\triangle,n}\},\{g_{\triangle,n}\}\subset
{\mathscr E}$. Thus, by the density of ${\mathscr
E}$ in ${\mathscr H}$ we can define
$$
\int_{\mathbb R}f(x){\mathscr L}^{a,b}(dx,t):=\lim_{n\to
\infty}\int_{\mathbb R}f_{\triangle,n}(x){\mathscr L}^{a,b}(dx,t)
$$
for any $f\in {\mathscr H}$, where $\{f_{\triangle,n}\}\subset
{\mathscr E}$ and
$$
\lim_{n\to \infty}f_{\triangle,n}(x)=f(x)
$$
in ${\mathscr H}$. These considerations are enough to prove the following theorem.
\begin{theorem}\label{th7.2}
Let $-1<b<0$, $a>-1$ and $-1<a+b<3$. For any $f\in {\mathscr H}$, the integral
$$
\int_{\mathbb R}f(x){\mathscr L}^{a,b}(dx,t)
$$
is well-defined in $L^2(\Omega)$ and the Bouleau-Yor type formula
\begin{equation}\label{sec6-eq6.3}
\bigl[f(B^{a,b}),B^{a,b}\bigr]^{(a,b)}_t=-\kappa_{a,b}\int_{\mathbb R}f(x)\mathscr{L}^{a,b}(dx,t)
\end{equation}
holds, almost surely, for all $t\in [0,T]$.
\end{theorem}

\begin{corollary}[Tanaka formula]
Let $-1<b<0$, $a>-1$ and $-1<a+b<3$. For any $x\in {\mathbb R}$ we have
\begin{align*}
(B^{a,b}_t-x)^{+}=(-x)^{+}+\int_0^t{1}_{\{B^{a,b}_s>x\}} dB^{a,b}_s
+\frac12{\mathscr L}^{a,b}(x,t),\\
(B^{a,b}_t-x)^{-}=(-x)^{-}-\int_0^t{1}_{\{B^{a,b}_s<x\}}
dB^{a,b}_s+\frac12{\mathscr L}^{a,b}(x,t),\\
|B^{a,b}_t-x|=|x|+\int_0^t{\rm sign}(B^{a,b}_s-x)dB^{a,b}_s+{\mathscr L}^{a,b}(x,t).
\end{align*}
\end{corollary}
\begin{proof}
Take $F(y)=(y-x)^{+}$. Then $F$ is absolutely continuous and
$$
F(x)=\int_{-\infty}^y1_{(x,\infty)}(y)dy.
$$
It follows from It\^o's formula~\eqref{sec7-eq7.3} and the
identity~\eqref{sec6-eq6.4} that
\begin{align*}
{\mathscr
L}^{a,b}(x,t)&=(\kappa_{a,b})^{-1} \bigl[1_{(x,+\infty)}(B^{a,b}),B^{a,b}\bigr]^{(a,b)}_t\\
&=2(B^{a,b}_t-a)^{+}-2(-x)^{+}-2 \int_0^t{1}_{\{B^{a,b}_s>x\}}dB^{a,b}_s
\end{align*}
for all $t\in [0,T]$, which gives the first identity. In the same
way one can obtain the second identity. By subtracting the last
identity from the previous one, we get the third identity.
\end{proof}

\begin{corollary}
Let $-1<b<0$, $a>-1$, $-1<a+b<3$ and let $f,f_1,f_2,\ldots\in {\mathscr H}$. If $f_n\to f$ in ${\mathscr H}$, as $n$ tends to infinity, we then have
\begin{align*}
\int_{\mathbb R}f_n(x){\mathscr L}^{a,b}(dx,t)\longrightarrow
\int_{\mathbb R}f(x){\mathscr L}^{a,b}(dx,t)
\end{align*}
in $L^2$, as $n$ tends to infinity.
\end{corollary}

According to Theorem~\ref{th7.1}, we get an analogue of the It\^o
formula (Bouleau-Yor type formula).
\begin{corollary}\label{cor7.1}
Let $-1<b<0$, $a>-1$, $-1<a+b<3$ and let $f\in {\mathscr H}$ be a left continuous function with right limits. If $F$ is an absolutely continuous function with $F'=f$ and the condition~\eqref{sec7-Ito-con1} is satisfied, then the following It\^o type formula holds:
\begin{equation}\label{sec7-eq7.11}
F(B^{a,b}_t)=F(0)+\int_0^tf(B^{a,b}_s)dB^{a,b}_s -\frac12\int_{\mathbb R}f(x){\mathscr L}^{a,b}(dx,t).
\end{equation}
\end{corollary}
Recall that if $F$ is the difference of two convex functions, then
$F$ is an absolutely continuous function with derivative of bounded
variation. Thus, the It\^o-Tanaka formula
\begin{align*}
F(B^{a,b}_t)&=F(0)+\int_0^tF^{'}(B^{a,b}_s)dB^{a,b}_s +\frac12\int_{\mathbb R}{\mathscr L}^{a,b}(x,t)F''(dx)\\
&\equiv F(0)+\int_0^tF^{'}(B^{a,b}_s)dB^{a,b}_s -\frac12\int_{\mathbb R}F'(x){\mathscr L}^{a,b}(dx,t)
\end{align*}
holds for all $-1<b<0$, $a>-1$ and $-1<a+b<3$.

%%%%%%%%%%%%%%%%%%%%%%%%%%%%%%%%%%%%%%%%%%%%%%%%%%%%%%%%%%%%%%%%%%%%%%%
%%%%%%%%%%%%%%%%%%%%%%%%%%%%%%%%%%%%%%%%%%%%%%%%%%%%%%%%%%%%%%%%%%%%%%%
%%%%%%%%%%%%%%%%%%%%%%%%%%%%%%%%%%%%%%%%%%%%%%%%%%%%%%%%%%%%%%%%%%%%%%%

\end{document}